\documentclass[a4paper]{amsart}
\usepackage{graphicx} 
\usepackage{amsfonts,amsmath,amssymb,amsthm}
\usepackage{mathtools}
\usepackage{xcolor}
\usepackage{url}
\usepackage{soul}
\usepackage[a4paper,top=3.5cm,bottom=1cm,left=2.5cm,right=2.5cm]{geometry}

\usepackage{amsfonts}
\usepackage{amsmath}
\usepackage{amsthm}
\usepackage{amssymb}
\usepackage{indentfirst}
\usepackage{xcolor}
\usepackage{tabularx}
\usepackage{graphicx}
\usepackage{esint}
\usepackage{dsfont}

\usepackage[T1]{fontenc}
\usepackage[utf8]{inputenc}
\usepackage[english]{babel}
\usepackage{microtype}
\usepackage{layaureo}
\usepackage{emptypage}
\usepackage{bm}
\usepackage{enumerate}
\usepackage{amsmath}
\usepackage{amsthm}
\usepackage{amssymb}
\usepackage{mathtools}
\usepackage{mathrsfs}
\usepackage{dsfont}
\usepackage{graphicx}
\usepackage{subfig}
\usepackage{caption}
\usepackage{cases}
\usepackage{lipsum}
\usepackage{afterpage}
\usepackage{hyperref}
\usepackage{comment}

\newcommand{\E}{\mathbb{E}}
\newcommand{\EE}{\mathbb{E}}

\newcommand{\R}{\mathbb{R}}

\newcommand{\de}{\mathrm{d}}
\newcommand{\Law}{\operatorname{Law}}

\newcommand{\cE}{\mathcal{E}}
\newcommand{\cP}{\mathcal{P}}
\newcommand{\cT}{\mathcal{T}}

\newcommand{\eps}{\varepsilon}
\newcommand{\di}{\mathrm{d}}
\newcommand{\Om}{\Omega}

\DeclareMathOperator{\argmin}{argmin}

\DeclarePairedDelimiter{\abs}{\lvert}{\rvert}
\DeclarePairedDelimiter{\norm}{\lVert}{\rVert}

\newcommand{\real}{\mathbb{R}}

\newcommand{\prob}{\mathbb{P}}

\newcommand{\ind}{\mathds{1}}

\newcommand{\tv}{\tilde{v}}
\newcommand{\tsigma}{\tilde{\sigma}}

\newcommand{\sF}{\mathscr{F}}

\newcommand{\diag}{\operatorname{diag}}
\newcommand{\trace}{\operatorname{tr}}

\newcommand{\green}[1]{\textcolor{orange}{#1}}

\numberwithin{equation}{section}

\newtheorem{definition}{Definition}[section]
\newtheorem{theorem}[definition]{Theorem}
\newtheorem{lemma}[definition]{Lemma}

\newtheorem{proposition}[definition]{Proposition}
\newtheorem{remark}[definition]{Remark}

\newtheorem{corollary}[definition]{Corollary}

\makeatletter
\@namedef{subjclassname@2020}{%
  \textup{2020} Mathematics Subject Classification}
\makeatother

\title[A general perspective on CBO methods with stochastic rate of information]{A general perspective on CBO methods\\with stochastic rate of information}

\author[S. Almi]{Stefano Almi}
\address[Stefano Almi]{Dipartimento di Matematica e Applicazioni ``R. Caccioppoli'', Unversit\`{a} di Napoli ``Federico II'', Via Cintia, 80126 Napoli, Italy. ORCID: 0000-0001-7308-221X.}
\email{stefano.almi@unina.it}

\author[A. Baldi]{Alessandro Baldi}
\address[Alessandro Baldi]{Dipartimento di Scienze Matematiche ``G. L. Lagrange'', Politecnico di Torino, Corso Duca degli Abruzzi, 24, 10129 Torino, Italy. ORCID: 0009-0006-4778-8431.}
\email{alessandro.baldi@polito.it}

\author[M. Morandotti]{Marco Morandotti}
\address[Marco Morandotti]{Dipartimento di Scienze Matematiche ``G. L. Lagrange'', Politecnico di Torino, Corso Duca degli Abruzzi, 24, 10129 Torino, Italy. ORCID: 0000-0003-3528-6152.}
\email{marco.morandotti@polito.it}

\author[F. Solombrino]{Francesco Solombrino}
\address[Francesco Solombrino]{Dipartimento di Scienze e Tecnologie Biologiche ed Ambientali, Universit\`{a} del Salento, Centro Ecotekne Pal.~B - S.P.~6, Lecce, Italy. ORCID: 0000-0003-4251-3851.}
\email{francesco.solombrino@unisalento.it}

\date{\today}

\keywords{Consensus-Based Optimization, stochastic information rate, concentration towards consensus}
\subjclass[2020]{34F05, 
                60K35, 
                93A16, 
                60H10, 
                90C56, 
                35Q84. 
                }

\begin{document}

\maketitle

\begin{abstract}
This paper studies a class of Consensus-Based Optimization (CBO) models featuring an additional stochastic rate of information, modeling the agents' knowledge of the environment and energy landscape. The well-posedness of the stochastic system is proved, together with its finite-particle approximation and the mean-field convergence to a kinetic PDE.  Particles are shown to concentrate around the consensus point under mild assumptions on the initial spatial distribution and initial level of knowledge. In particular, the analysis unveils that a positive, however small, initial level of knowledge is enough for convergence to consensus to happen. The framework presented is general enough to include the first instances of CBO proposed in the literature.
\end{abstract}

\section{Introduction}

Consensus-Based Optimization (CBO) is a recent stochastic optimization paradigm inspired by collective dynamics in multi-agent systems. It provides a powerful framework for tackling non-convex optimization problems in high-dimensional spaces. CBO algorithms operate by evolving a population of interacting particles whose dynamics are driven by an interplay of consensus formation, stochastic exploration, and a bias toward low function values. These dynamics naturally lend themselves to a probabilistic and mean-field theoretical analysis.

Initial motivations for CBO arise from a variety of heuristic swarm intelligence algorithms, such as Particle Swarm Optimization (PSO) \cite{GrassiPareschi, KennedyEberhart1995}, Simulated Annealing \cite{Aarts, Kirkpatrick1983}, Differential Evolution~\cite{StornPrice1997}, Evolutionary Programming~\cite{BackFogel, Fogel}, or Genetic Algorithms~\cite{Holland}. These methods have demonstrated empirical success but often lack rigorous analytical foundations or suffer from scalability issues. In contrast, the CBO framework is deeply connected with the mathematical theory of mean-field particle systems and kinetic equations, which have found applications in statistical physics, control theory, and mathematical biology \cite{Carrillo2010,FornasierSolombrino2014, Rossi-review, Toscani2006}.

The recent works~\cite{Carrilloetal, Pinnau} proposed a mathematically grounded CBO scheme based on stochastic differential equations. Each of the \( N \) particles evolves according to
\begin{equation} \label{eq:sde}
    \di X_t^i = -\lambda (X_t^i - \overline{X}_t^n)\,\di t + \varsigma \|X_t^i - \overline{X}_t^n \|\,\di B_t^i,
\end{equation}
where \( \lambda > 0 \) is the consensus rate, \( \varsigma > 0 \) is the noise strength, and \( B_t^i \) are independent Brownian motions. The consensus point \( \overline{X}_t^n \) is defined via a Gibbs-weighted average
\begin{equation} \label{eq:weighted_mean}
    \overline{X}_t^n = \frac{\sum_{j=1}^N X_t^j \exp(-n \mathcal{E} (X_t^j))}{\sum_{j=1}^N \exp(-n \mathcal{E} (X_t^j))},
\end{equation}
with $\mathcal{E} \colon \mathbb{R}^d \to \mathbb{R}$ the objective function and $n \in \mathbb{N}$ a concentration parameter. Notice that in~\cite{Carrilloetal} $n$ is replaced by a $\alpha>0$. We restrict our presentation to integer coefficients for clarity of exposition.

As \( N \to \infty \), the empirical measure \( \mu_t^N = \frac{1}{N} \sum_{i=1}^N \delta_{X_t^i} \) converges (formally) to a probability density \( \rho(x,t) \), satisfying a nonlinear Fokker--Planck equation (see~\cite{HuangQiu})
\begin{equation}
 \label{eq:fp}
    \partial_t \rho = \nabla \cdot \left( \lambda (x - \bar{x}_t^n)\rho \right) + \frac{\varsigma^2}{2} {\rm div} \left( \|x - \bar{x}_t^n \|^2 \nabla \rho \right),
\end{equation}
with
\begin{equation}
\label{e:intro100}
    \bar{x}_t^n = \frac{\int x e^{-n \mathcal{E} (x)} \rho(x,t)\,dx}{\int e^{-n \mathcal{E} (x)} \rho(x,t)\,dx}.
\end{equation}
This formulation enables the use of tools from PDEs, probability theory, and optimal transport to study the long-time behavior of the system. In particular, the concentration of \( \rho \) around global minimizers is justified using the Laplace principle \cite{DemboZeitouni1998,DupuisEllis1997}, which implies
\[
    \lim_{n \to \infty} \bar{x}_t^n = \argmin_{x \in \R^{d}} \mathcal{E}(x),
\]
under suitable regularity and coercivity assumptions on $\mathcal{E}$.

Building on this foundation, Fornasier, Klock, and Riedl \cite{FKR} provided a rigorous convergence analysis of the CBO method in the mean-field setting. Their work proves that the mean-field solution \( \rho(x,t) \) concentrates globally on minimizers of $\mathcal{E}$ as \( t \to + \infty \). A central theme in this analysis is the role of the \emph{information rate} \( \lambda \), which governs the balance between convergence speed and exploration. Larger values of \( \lambda \) lead to faster consensus formation but may suppress the diversity needed to escape local minima. This trade-off echoes similar phenomena in stochastic gradient algorithms \cite{Cheng2018, Raginsky2017} and in control-based optimization dynamics \cite{Benamou2015,Pavliotis2014}. Further extension of the CBO model have been discussed, for instance, in~\cite{Herty-Pareschi, CarrilloVaes} for constrained optimization, or in~\cite{CarrilloJin, Sunnen1, Sunnen2, Sunnen3} for minimization on hypersurfaces and applications to machine learning.

The works \cite{Carrilloetal} and \cite{FKR} provide a rigorous and rich mathematical theory for CBO, bridging stochastic optimization, collective dynamics, and PDE analysis. Their framework offers not only practical algorithms but also deep insights into the interplay between randomness, information aggregation, and global convergence. Notice that in those contributions, as in our own paper, we will assume the objective energy~$\mathcal{E}$ to have a unique global minimizer (we refer to~\cite{FornasierSun} for some results concerning multiple minimizers).

This leads us to the purpose of our present contribution, which has a twofold perspective:
\begin{itemize}
\item The above results are inserted in a general, abstract toolbox capturing the main features of this approach;
\item An additional, in our opinion quite relevant, modeling possibility is considered, as the \emph{information rate} \( \lambda \) is  itself  allowed to be a time-evolving unknown of the underlying SDE system.
\end{itemize}

The particle system we consider, generalising  the stochastic counterpart of \eqref{eq:sde}, is formulated as a McKean-Vlasov SDE of the form
\begin{equation}\label{e:intro101}
    \begin{dcases}
        \de X_t = \nu\,v_{\rho_t}(X_t,\Lambda_t)\,\de t + \varsigma  \| v_{\rho_t}(X_t,\Lambda_t) \| \,\de B_t\,, \\
        \de \Lambda_t = \cT_{\Sigma_t}(X_t,\Lambda_t)\,\de t,
    \end{dcases}
\end{equation}
with
\begin{gather*}
    \Sigma_t = \Law(X_t,\Lambda_t), \qquad \rho_t = \Law(X_t)\,.
\end{gather*}
Here $\nu,\varsigma$ are non-negative parameters and
\begin{equation*}
    v_\mu(x,\lambda) \coloneqq \lambda( f_{n}(\mu) -x  )  + (1-\lambda) ( e(\mu) - x) \,,
\end{equation*}
with $e(\mu)$ being the expected value of a probability distribution $\mu$. The key features of the above equations are encoded by two terms:
\begin{itemize}
\item $f_{n} (\mu)$ is a nonlocal drift term generalizing the Gibbs energies considered in~\cite{Carrilloetal, FKR} (cf.~\eqref{e:intro100}). We show that, instead of the weighted barycenter $\bar{x}^{n}_{t}$, one can consider a more general class of dependencies on~$\mu$ and~$\mathcal{E}$, provided some abstract properties are satisfied (cf.~$(f1)$, $(f2)$, and~$(f3)$ below).

\item $\cT$ is the generator of a Markov-type jump process governing the evolution of the information rate $\Lambda_{t}$ of a single agent, which may  grow or decrease in time, depending on the interaction with the rest of the population, and hence on the global state of the system~$\Sigma_{t}$.
\end{itemize}

The interpretation of system~\eqref{e:intro101} is that an informed agent is expected to have $\lambda \sim 1$ and to move towards the temporary consensus $f_{n} (\mu)$, which favors concentration towards the minimizer of~$\mathcal{E}$. On the contrary, in presence of a low rate of information $\lambda \sim 0$, the lack of knowledge pushes the agent to follow the average position of the crowd $e(\mu)$. The motion is affected by some white noise, which decreases with the increase of the information rate in the system~\eqref{e:intro101}, and models the random behavior of a non-informed crowd. The effect of noise is further tuned by~$\varsigma>0$. Such parameter will be assumed to be sufficiently small in our convergence analysis, with bounds only depending on the space dimension~$d$.  The PDE counterpart of~\eqref{e:intro101} is a Fokker-Planck equation
  \begin{equation}
    \label{e:intro102}
        \partial_{t} \Sigma_{t} - \frac{\varsigma^{2}}{2} \Delta_{x} ( \| v_{\rho_{t}} (x, \lambda) \|^{2} \Sigma_{t}) = {\rm div} \Big[ \big( \nu \, v_{\rho_{t}} (x, \lambda) \, , \,  \cT_{\Sigma_{t}} (x, \lambda)  \big) \Sigma_{t} \Big]\,,
    \end{equation}
 on the product space $\R^{d} \times [0, 1]$. Differently from~\cite{FKR}, we find it convenient to work directly on~\eqref{e:intro101} to recover our well-posedness and convergence results, which we describe below.
 
 The main contributions of the present paper concern
 \begin{itemize}
 \item[$(i)$] A global well-posedness theory for system~\eqref{e:intro101}, which extends to the stochastic setting the work of~\cite{AFMS, MorSol} for multi-agent multi-label systems. As it was there, the state space is a convex subset of a Banach space, hence well-posedness of evolutions can be obtained only under suitable assumptions on the operator~$\cT$, in the spirit of~\cite{BM, Brezis, D'Onofrio-Hernandez}. We refer to the conditions $(T1)$ and $(T2)$ below.
 
 \item[$(ii)$] Our main Theorem~\ref{t:systems are close} shows that under the assumptions reported on~$f_{n}$ and~$\cT$, if some information is present in the system at initial time, then spatial concentration around the minimizer of~$\mathcal{E}$ will take place in finite time. This generalizes the main result of~\cite{FKR} to other choices of the drift term $f_{n}$ and to the stochastic nature of the information rate $\lambda$. In Section~\ref{s:application} we further show that their results are indeed a particular case of our analysis, as their case study complies with our abstract setting.
 \end{itemize}
In our convergence analysis we assume, without loss of generality, that the targeted global minimizer is $x=0$, with $\mathcal{E}(0) = 0$. We first consider an auxiliary system, where the drift term~$f = 0$
\begin{equation}\label{e:intro103}
\begin{cases}
\di X_{t} = (- X_{t} + ( 1 - \Lambda_{t}) e(\mu_{t}) ) ) \di t + \varsigma \| X_{t} - ( 1 - \Lambda_{t}) e (\mu_{t} ) \| \di B_{t}\,,\\[1mm]
\di \Lambda_{t} = \cT_{\Psi_{t}}(X_{t},\Lambda_{t})\, \di t\,,
\end{cases}
\end{equation}
with $\Psi_{t} = \Law (X_{t}, \Lambda_{t})$ and $\mu_{t} = \Law (X_{t})$. By means of Gr\"onwall-type estimates and It\^o calculus, we show that if the rate of information~$\Lambda_{t}$ is not vanishing too quickly as $t \to + \infty$, that is,
\begin{equation}
\label{e:intro105}
\int_{0}^{+\infty} \E (\Lambda_{t}) \,\di t = +\infty\,,
\end{equation}
then solutions to~\eqref{e:intro103} converge almost surely to the minimizer $x=0$ as $t \to + \infty$. In Theorem~\ref{t:concentration2}, we indeed prove that under our assumption on~$\cT$, property~\eqref{e:intro105} can be guaranteed, entailing the desired convergence.

To recover Theorem~\ref{t:systems are close}, it remains to compare the behavior of the two systems~\eqref{e:intro102} and~\eqref{e:intro103}, for large values of~$n$. The crucial assumption is that $f_{n} (\mu)$ converges to $0$ uniformly, provided that some upper bound on the second moment and some lower bound on the mass around $0$ are ensured along the motion. The prototypical example of~\cite{Carrilloetal, FKR} fits into this framework, as shown in Section~\ref{s:application}, with the required upper and lower bounds on~$\rho_{t}$ recovered in Theorem~\ref{t:well-posedness} and Proposition~\ref{e:mass-ball-r}, respectively.

\medskip

\noindent{\bf Outlook.} We manage to provide a comprehensive abstract toolbox for multi-agent Consensus Based Optimization, generalizing the results~\cite{FKR} to the case of a stochastic varying rate of information. Desirable developments of of our analysis take into account the following directions:
\begin{itemize}
\item The addition of noisy effects acting directly on $\Lambda_{t}$\,, still guaranteeing the convex constraint $\Lambda_{t} \in [0, 1]$.

\item In a more general multi-population setting, $\Lambda_{t}$ can be interpreted as a stochastic label representing the probability that an agent has to belong to a certain subset of the population. This is a typical framework for population dynamics~\cite{During, Tosin, Toscani2006}. It would be interesting to extend our model to encompass multiple, possibly conflicting populations.

\item While in our analysis deals with a single objective with a single target minimizer, the study of multi-target and/or multi-objective minimization is of interest for application purposes~\cite{Herty-Pareschi-2, FornasierSun, Totzeck-multiobjective}.
\end{itemize}

\section{Preliminaries}
\label{s:preliminaries}
\subsection{Notation}
For $x \in \R^{d}$ and $r>0$, we denote by $B_{r} (x)$ the open ball of center $x$ and radius $r$. If $x=0$, we will adopt the short-hand notation $B_{r} = B_{r} (0)$. We denote by~$\| \cdot \|$ the Euclidean norm in~$\R^{d}$ ($d>1$) and by~$| \cdot|$ the modulus in~$\R$.
For $p \geq 1$ we denote by $\cP_{p}(\R^{d})$ the set of probability measures $\mu$ over $\R^{d}$ with finite $p$-th moment
\begin{displaymath}
m_{p} (\mu) \coloneqq \bigg( \int_{\R^{d}} \| x\|^{p} \, \di \mu(x) \bigg) ^{\frac{1}{p}} <+\infty\,.
\end{displaymath}
Such space is endowed with the classical $p$-Wasserstein distance, which we denote by $W_{p}$\,.

\subsection{A well-posedness result}

We start by recalling a well-posedness result for the system
\begin{equation}
\label{e:system-SDE-aux}
    \begin{dcases}
        \de X_t = \nu w_{\rho_t}(X_t,\Lambda_t)\,\de t + \varsigma \norm{w_{\rho_t}(X_t,\Lambda_t)}\,\de B_t\,, \\
        \de \Lambda_t = \cT_{\Sigma_t}(X_t,\Lambda_t)\,\de t,\\
         \Sigma_t = \Law(X_t,\Lambda_t), \\
         (X_{0}, \Lambda_{0}) = (\widehat{X}_{0}, \widehat{\Lambda}_{0}),\\
         \rho_t = (\pi_x)_\sharp \Sigma_t = \Law(X_t)\,.
    \end{dcases}
\end{equation}
Here, $(\widehat{X}_{0}, \widehat{\Lambda}_{0})$ are suitable initial conditions, $\pi_x \colon \R^d\times[0,1] \to \R^d$ is the canonical projection on the first factor, $\nu,\varsigma$ are non-negative parameters. We make the following assumptions on the velocity field $w_{\mu} \colon \R^{d} \times [0, 1] \to \R^{d}$ and on the operator~$\cT_{\Psi} \colon \R^{d} \times [0, 1] \to \R$, defined for $\mu \in \cP_{1} (\R^{d})$ and $\Psi \in \cP_{1} (\R^{d} \times [0, 1])$:

\begin{enumerate}
\item[$(w1)$] There exists $L_{w}>0$ such that for every $\Psi_{1}, \Psi_{2} \in \cP_{1} (\R^{d}\times[0,1])$, every $x_{1}, x_{2} \in \R^{d}$, and every $\lambda_{1}, \lambda_{2} \in [0, 1]$
\begin{equation*}
\| w_{\mu_{1}} (x_{1}, \lambda_{1}) - w_{\mu_{2}} (x_{2}, \lambda_{2}) \| \leq L_{w} \big( \| x_{1} - x_{2}\| + |\lambda_{1} - \lambda_{2}| + W_{1} (\Psi_{1} , \Psi_{2}) \big) ,
\end{equation*}
where $\mu_i \coloneqq (\pi_x)_\sharp\Psi_i\,$, for $i = 1,2$.
\item[$(T1)$] The exists $L_\cT>0 $ such that for all $\Psi_1,\Psi_2 \in \cP_1(\R^d\times[0,1])$ and $(x_1,\lambda_1), (x_2,\lambda_2) \in \R^d\times[0,1]$
\begin{equation*}    
\abs{\cT_{\Psi_1}(x_1,\lambda_1) - \cT_{\Psi_2}(x_2,\lambda_2)} \le L_\cT(\norm{x_1 - x_2} + \abs{\lambda_1 - \lambda_2} + W_1 (\Psi_1,\Psi_2));
\end{equation*}
\item[$(T2)$] There exists $\theta > 0$ such that for all $\Psi \in \cP_1(\R^d\times[0,1])$ and $(x,\lambda) \in \R^d \times [0,1]$
\begin{equation*}
\lambda + \theta \cT_\Psi(x,\lambda) \in [0,1].
\end{equation*}
\end{enumerate}

\begin{remark}
Notice that conditions~$(w1)$, $(T1)$, and $(T2)$ are inspired by the corresponding assumptions in~\cite{AFMS, BM, MorSol}. In the present setting, an adaptation is required, which accounts for extra finite moments. This is crucial for the analysis of CBO-type of systems (see~\cite{Carrilloetal, FKR}).
\end{remark}

\begin{proposition}
\label{p:existence-truncated}
Let $T > 0$, let $B \colon \Omega \to C([0,T],\R^d)$ be a $d$-dimensional standard Brownian motion defined on a filtered probability space $\bigl(\Omega, \sF, (\sF_t)_{t\in[0,T]},\prob\bigr)$, $\widehat{X}_0\colon\Omega\to\R^d$ an $\sF_0$-measurable random variable belonging to $L^p(\Omega,\sF,\prob)$ with $p \geq 2$, and $\widehat{\Lambda}_0\colon\Omega\to [0,1]$ an $\sF_0$-measurable random variable. Under the assumptions $(w1)$, $(T1)$, and $(T2)$, there exists an $(\sF_t)_t$-adapted continuous stochastic process $Y = (X,\Lambda ) \colon \Omega \to C\bigl([0,T],\R^d\times [0,1]\bigr)$ which satisfies the equation
\begin{equation}
\label{e:integral-equation}
\begin{dcases}
    X_t =  \widehat{X}_0 + \int_0^t \nu \, w_{\rho_{s}}(X_s,\Lambda_s)\,\de s + \int_0^t\varsigma \norm{w_{\rho_{s}} (X_s,\Lambda_s)}\,\de B_s\,, \\
    \Lambda_t = \widehat{\Lambda}_0 +\int_0^t \cT_{\Sigma_s}(X_s,\Lambda_s)\,\de s,
\end{dcases}
\end{equation}
with $\Sigma_t = \Law(X_t, \Lambda_{t})$ (and $\rho_t = (\pi_x)_\sharp \Sigma_t = \Law(X_t)$), $\prob$-almost surely and for all $t \in [0,T]$. The solution satisfies the estimate
\begin{equation}
\label{e:4-moment}
    \E\biggl[\sup_{t\in[0,T]}\norm{Y_t}_{\R^d\times[0,1]}^p\biggr] \le C \biggl(1 + \E\bigl[\norm{\widehat{X}_0}^p\bigr]\biggr)
\end{equation}
for a constant $C = C(\rho , T) >0$. 
Moreover, the map $t \mapsto \Sigma_t$ belongs to $C([0,T],(\cP_p(\R^d \times [0,1]),W_p))$ 
and the solution is pathwise unique in the class of (strong) solutions such that the curve $t \mapsto m_1(\rho_t)$ is bounded over $[0,T]$.

\end{proposition}

\begin{proof}
The well-posedness result and estimate \eqref{e:4-moment} 
follow directly from~\cite[Theorem~4.2]{BM}, of which system \eqref{e:system-SDE}, under the hypotheses $(w1)$, $(T1)$, and $(T2)$, constitutes a particular case of application.
\end{proof}

It will be useful to have also the following a priori estimate concerning~\eqref{e:system-SDE-aux}, obtained as a result of the following condition:
\begin{enumerate}
\item[$(w2)$] There exists $M_{w} >0$ such that for every $x \in \R^{d}$, every $\lambda \in [0, 1]$, and every $\mu \in \cP_{1} (\R^{d})$
\begin{displaymath}
\| w_{\mu} (x, \lambda) \| \leq M_{w} \big( 1 + \| x \| + m_{1} (\mu) \big).
\end{displaymath}
\end{enumerate}

\begin{proposition} 
\label{p:a-priori-estimate-SDE}
Let $T > 0$,  let $B \colon \Omega \to C([0,T],\R^d)$ be a $d$-dimensional standard Brownian motion defined on a filtered probability space $\bigl(\Omega, \sF, (\sF_t)_{t\in[0,T]},\prob\bigr)$, $\widehat{X}_0\colon\Omega\to\R^d$ an $\sF_0$-measurable random variable belonging to $L^p(\Omega,\sF,\prob)$ with $p \geq 2$, and $\widehat{\Lambda}_0\colon\Omega\to [0,1]$ an $\sF_0$-measurable random variable. Under the assumption $(w2)$, any strong solution $Y = (X, \Lambda)$ to system~\eqref{e:system-SDE-aux} such that the curve $t \mapsto m_1(\rho_{t})$ is bounded over $[0, T]$ satisfies the following \emph{a priori} estimate
\begin{equation*}
    \E\biggl[\sup_{t\in[0,T]}\norm{Y_t}_{\R^d\times[0,1]}^p\biggr] \le C \biggl(1 + \E\bigl( \norm{X_0}^p\bigr) \biggr),
\end{equation*}
for a constant $C = C( \rho, T)>0$.
\end{proposition}

\begin{proof} Condition $(w2)$ and the boundedness of the map $t \mapsto m_1(\rho_t)$ over $[0,T]$ imply the following (uniform in time) sublinearity estimate: for all $(x,\lambda) \in \R^d \times [0,1]$ and $t \in [0,T]$,
\begin{equation*}
    \| w_{\rho_t} (x, \lambda) \| \leq M_{w} \big( 1 + \| x \| + m_{1} (\rho_t) \big) \le C_{\rho}(1 + \norm{x}) ,
\end{equation*}
where $C_\rho>0$ is a suitable constant depending on $M_w$ and on the curve $\rho$. This allows us to argue as in \cite[Proposition~4.5]{BM} (see also Proposition 3.14 in the same work).
\end{proof}

\section{Well-posedness of a Consensus-Based Optimization model}
\label{s:well-posedness}
In this section we extend Proposition~\ref{p:existence-truncated} in order to account for a particular locally Lipschitz velocity~$w$. More precisely, we investigate the well-posedness of the system 
\begin{equation}
\label{e:system-SDE}
    \begin{dcases}
        \de X_t = \nu\, \big( -X_{t} + \Lambda_{t} f(\rho_{t}) + (1 - \Lambda_{t}) e(\rho_{t} ) \big) \,\de t + \varsigma \norm{ -X_{t} + \Lambda_{t} f(\rho_{t}) + (1 - \Lambda_{t}) e(\rho_{t}) }\,\de B_t\,, \\
        \de \Lambda_t = \cT_{\Sigma_t}(X_t,\Lambda_t)\,\de t,\\
         \Sigma_t = \Law(X_t,\Lambda_t), \\
         (X_{0}, \Lambda_{0}) = (\widehat{X}_{0}, \widehat{\Lambda}_{0}),\\
         \rho_t = (\pi_x)_\sharp \Sigma_t = \Law(X_t)\,.
    \end{dcases}
\end{equation}
where $\nu,\varsigma$ are non-negative parameters,  $f \colon \cP_{1} (\R^{d}) \to \R^{d}$ is  suitable given map, and
\begin{align}
    e(\mu) & \coloneqq  \int_{\R^d} x\,\de \mu,\quad \text{ for all } \mu \in \cP_1(\R^d).\label{e:mean-operator}
\end{align}
To shorten notation, we set
\begin{displaymath}
  v_\mu(x,\lambda)  \coloneqq -x + \lambda f(\mu) + (1-\lambda)e(\mu) \quad \text{ for all } \mu \in \cP_1(\R^d), \, x \in \R^{d}, \, \lambda \in [0, 1].
\end{displaymath}
The structure of the SDE system~\eqref{e:system-SDE}--\eqref{e:mean-operator} resembles the one typically considered in CBO models (cf.~\cite{Carrilloetal, FKR}), that we will consider in more detail in Section~\ref{s:application}. 

We make the following assumptions on $f$:
\begin{enumerate}
    \item[$(f1)$] For every $R > 0$, there exists $L_{f,R} > 0$ such that, for all $\mu_1,\mu_2 \in \cP_1(\R^d)$ with
    \begin{equation*}
        m_{1} (\mu_{1}) , \, m_{1} (\mu_{2}) \le R, 
    \end{equation*}
    the following estimate holds:
    \begin{equation}
        \norm{f(\mu_1) - f(\mu_2)} \le L_{f,R} W_1(\mu_1,\mu_2).
    \end{equation}
    \item[$(f2)$]
    There exists a constant $M_{f} \ge 0$ such that, for all $\mu \in \cP_1(\R^d)$, 
    \begin{equation*}
        \norm{f(\mu)} \le M_{f} \big(1 + m_{1} (\mu) \big).
    \end{equation*} 
   \end{enumerate}
As a consequence of $(f1)$--$(f2)$, the velocity field $v$ satisfies the following properties.
\begin{proposition} 
\label{p:some-estimates}
The following inequalities hold:
\begin{enumerate}
\item[$(v1)$] For all $\mu \in \cP_1(\R^d)$ and $(x_1,\lambda), (x_2,\lambda) \in \R^d\times[0,1]$,
\begin{equation*}
    \norm{v_\mu(x_1,\lambda) - v_\mu(x_2,\lambda)} = \norm{x_1 - x_2};
\end{equation*}
\item[$(v2)$] for all $\mu \in \cP_1(\R^d)$ and $(x,\lambda_1), (x,\lambda_2) \in \R^d\times[0,1]$,
\begin{equation*}
    \norm{v_\mu(x,\lambda_1) - v_\mu(x,\lambda_2)} \le (\norm{f(\mu)} + \norm{e(\mu)})\abs{\lambda_1 - \lambda_2};
\end{equation*}
\item[$(v3)$] for all $\mu_1, \mu_2 \in \cP_1(\R^d)$ and $(x,\lambda) \in \R^d\times[0,1]$,
\begin{equation*}
    \norm{v_{\mu_1}(x,\lambda) - v_{\mu_2}(x,\lambda)} \le \norm{f(\mu_1) - f(\mu_2)} + \norm{e(\mu_1) - e(\mu_2)}.
\end{equation*}
\end{enumerate}

Moreover, $v_{\mu}$ satisfies condition~$(w2)$. 
\end{proposition}

\begin{proof} Inequalities ($v1$), ($v2$), and $(v3)$ readily follow form a direct computation. As for condition $(w2)$, we notice that by $(f2)$ and by definition of~$e(\mu)$ in~\eqref{e:mean-operator} we have that
\begin{align*}
\| v_{\mu} (x, \lambda) \| \leq \| x\| + \| f(\mu)\| + \| e(\mu)\| \leq \| x \| + M_{f} \big( 1 + m_{1} (\mu) \big) + m_{1} (\mu) \leq (M_{f} + 1) \big( 1 + \| x \| + m_{1} (\mu) \big) \,,
\end{align*}
which concludes the proof.
\end{proof}

\begin{remark}
\label{r:e-lip}
We observe that $e$ is Lipschitz continuous with respect to the $1$-Wasserstein distance. Indeed, for every $\mu_{1}, \mu_{2} \in \cP_{1} (\R^{d})$ and every transport plan $\gamma \in \cP_{1}(\R^{d} \times \R^{d})$ between $\mu_{1}$ and $\mu_{2}$ we have that
\begin{equation*}
    \norm{e(\mu_1) - e(\mu_2)} = \bigg \| \int_{\R^{d}} x \, \de \mu_{1} (x) - \int_{\R^{d}} y \de \mu_{2} (y) \bigg\| = \bigg\| \int_{\R^{d} \times \R^{d}} (x - y) \, \de \gamma(x, y)\bigg\|  \leq \int_{\R^{d} \times \R^{d}} \norm{x - y}\, \de \gamma (x, y)\,.
\end{equation*}
Hence, we infer the desired Lipschitz continuity of~$e$.
\end{remark}

We now state the main result of this section, extending Proposition~\ref{p:existence-truncated}.

\begin{theorem}
\label{t:well-posedness}
Given $p \geq 2$, $T > 0$, let $B \colon \Omega \to C([0,T],\R^d)$ be a $d$-dimensional standard Brownian motion defined on a filtered probability space $\bigl(\Omega, \sF, (\sF_t)_{t\in[0,T]},\prob\bigr)$, $\widehat{X}_0\colon\Omega\to\R^d$ an $\sF_0$-measurable random variable belonging to $L^p(\Omega,\sF,\prob)$, and $\widehat{\Lambda}_0\colon\Omega\to [0,1]$ an $\sF_0$-measurable random variable. Assume that $(f1)$--$(f2)$ and $(T1)$--$(T2)$ hold. Then, there exists an $(\sF_t)_t$-adapted continuous stochastic process
$Y = (X,\Lambda) \colon \Omega \to C\bigl([0,T],\R^d\times [0,1]\bigr)$ which satisfies the equation
\begin{equation*}
\begin{dcases}
    X_t =  X_0 + \int_0^t \nu\, v_{\rho_s}(X_s,\Lambda_s)\,\de s + \int_0^t\varsigma \norm{v_{\rho_s}(X_s,\Lambda_s)}\,\de B_s\,, \\
    \Lambda_t = \Lambda_0 +\int_0^t \cT_{\Sigma_s}(X_s,\Lambda_s)\,\de s,
\end{dcases}
\end{equation*}
with $\Sigma_t = \Law(Y_t)$ (and $\rho_t = (\pi_x)_\sharp \Sigma_t = \Law(X_t)$), $\prob$-almost surely and for all $t \in [0,T]$. The solution satisfies the estimate
\begin{align}
\label{e:moment-estimate}
   \sup_{t \in [0, T]}  \E\big( \norm{Y_t}_{\R^d\times[0,1]}^2\big)  & \le C \biggl(1 + \E\bigl( \norm{X_0}^2 \bigr) \biggr)\,,\\
   \label{e:momento4}
      \E\biggl[\sup_{t \in [0, T]} \norm{Y_t}_{\R^d\times[0,1]}^p\biggr] & \le C_{\rho, T}  \biggl(1 + \E\bigl( \norm{X_0}^p \bigr) \biggr)
\end{align}
for some positive constants~$C = C(T)$ only depending on~$T$ and $C_{\rho, T}$ depending on~$\rho$ and~$T$. Moreover, the map $t \mapsto \Sigma_t$ belongs to $C([0,T],(\cP_p(\R^d \times [0,1]),W_p))$ and  the solution is pathwise unique in the class of (strong) solutions such that the curve $t \mapsto m_{1} (\rho_{t})$ is bounded over $[0,T]$.
\end{theorem}

In order to prove Theorem~\ref{t:well-posedness}, we proceed with a truncation argument. 
For this purpose, we introduce some further notation. For any $R > 0$ we introduce a \emph{cut-off} function $\eta_R \in C^{\infty}_c(\R)$ such that $0 \le \eta_R \le 1$, $| \eta'_{R}| \leq 2$, and
\begin{equation*}
\eta_R(z) =
\begin{dcases}
    1 &\text{if } z \le R, \\
    0 &\text{if } z \ge R + 1\,.
\end{dcases}
\end{equation*}
An explicit example is given by the function
\begin{equation*}
    \eta_R(z) = \frac{h(R+1 - z)}{h(R+1 - z) + h(z - R)}
\end{equation*}
where $h(t) = e^{-1/t}\ind_{(0,+\infty)}(t)$. By means of this cut-off function we construct $\varphi_R \colon \cP_1 (\R^d) \to \R$ as
\begin{equation*}
    \varphi_R(\mu) \coloneqq \eta_R \big( m_{1} (\mu) \big),\quad \text{for all } \mu \in \cP_1(\R^d).
\end{equation*}
We notice that $\varphi_{R}$ satisfies $0 \le \varphi_R \le 1$ and for every $\mu \in \cP_{1} (\R^{d})$
\begin{equation*}
\varphi_R(\mu) =
\begin{dcases}
    1 &\text{if }  m_{1} (\mu) \le R, \\
    0 &\text{if } \ m_{1} (\mu) \ge R + 1.
\end{dcases}
\end{equation*}

\begin{proposition} 
\label{p:lip-phi}
The function $\varphi_R$ is Lipschitz continuous on $(\cP_1(\R^d),W_1)$. More precisely, for any $R > 0$, we have that
\begin{equation*}
    \abs{\varphi_R(\mu_1) - \varphi_R(\mu_2)} \le 2 W_1(\mu_1,\mu_2)
\end{equation*}
for all $\mu_1,\mu_2 \in \cP_1(\R^d)$.
\end{proposition}

\begin{proof} 
The function $\eta_R$ is $2$-Lipschitz since $|\eta'_{R}|\leq 2$. Moreover, the map $ \mu \mapsto  m_{1} (\mu)$ is $1$-Lipschitz on $\cP_1(\R^d)$ with respect to the Wasserstein distance $W_1$ thanks to \cite[Proposition 7.29]{Villani}. Hence, the claim readily follows by composition of Lipschitz functions.
\end{proof}

For all $R > 0$ we introduce the truncated field
\begin{equation*}
    \tv^R_\mu(x,\lambda) = -x + \lambda\varphi_R(\mu)f(\mu) + (1-\lambda)\varphi_R(\mu)e(\mu), \quad \text{ for all } \mu \in \cP_1(\R^d), (x,\lambda) \in \R^d \times [0,1].
\end{equation*}

\begin{proposition} 
\label{p-vR-lip}
Assume $(f1)$--$(f2)$. Then, for every $R > 0$ there exists a constant $L_R \ge 0$ such that
\begin{equation*}
    \norm{\tv^R_{\mu_1}(x_1,\lambda_1) - \tv^R_{\mu_2}\bigl(x_2,\lambda_2)} \le L_R(\norm{x_1 - x_2} + \abs{\lambda_1 - \lambda_2} + W_1(\mu_1,\mu_2)\bigr)
\end{equation*}
for all $\mu_1,\mu_2 \in \cP_1(\R^d)$ and $(x_1,\lambda_1), (x_2,\lambda_2) \in \R^d \times [0,1]$.
\end{proposition}

\begin{proof} 
It is enough to show the Lipschitz continuity of $(\lambda, \mu) \mapsto \lambda \varphi_{R} (\mu) f(\mu)$ and of $(\lambda, \mu) \mapsto ( 1 - \lambda) \varphi_{R} (\mu) e(\mu)$  in $[0, 1] \times \cP_{1} (\R^{d})$. We focus on the first map. For $\mu_{1}, \mu_{2} \in \cP_{1} (\R^{d})$ and $\lambda_{1}, \lambda_{2} \in [0, 1]$, we distinguish three cases. If $m_{1} (\mu_{i}) \leq R+1$ for $i=1, 2$, we estimate
\begin{align}
\label{e:triangular-ineq}
 \| \lambda_{1}  \varphi_{R} (\mu_{1}) f(\mu_{1}) & - \lambda_{2} \varphi_{R} (\mu_{2}) f(\mu_{2})\| \leq   | \lambda_{1} | \, | \varphi_{R}(\mu_{1}) - \varphi_{R} (\mu_{2}) | \, \| f(\mu_{1})\| 
 \\
 &
 + |\lambda_{1} - \lambda_{2}| \, |  \varphi_{R} (\mu_{2}) | \, \| f (\mu_{1})\| + |\lambda_{2}| \, |  \varphi_{R} (\mu_{2}) | \, \| f (\mu_{1}) - f(\mu_{2}) \|\,. \nonumber
\end{align}
In view of $(f1)$--$(f2)$ and of Proposition~\ref{p:lip-phi}, we get that
\begin{align*}
\| \lambda_{1} & \varphi_{R} (\mu_{1}) f(\mu_{1}) - \lambda_{2} \varphi_{R} (\mu_{2}) f(\mu_{2})\| 
\\
&
\leq 2M_{f} ( 2 + R)  W_{1} (\mu_{1}, \mu_{2}) + M_{f} ( 2 + R) | \lambda_{1} - \lambda_{2}| + L_{f, R+1} W_{1} (\mu_{1}, \mu_{2}) \,.
\end{align*}

In the second case, we assume that $m_{1} (\mu_{2}) >R+1$ and $m_{1} (\mu_{1}) \leq R+1$. Then, starting from~\eqref{e:triangular-ineq} we notice that $\varphi_{R} (\mu_{2}) = 0$, so that
\begin{align*}
\| \lambda_{1} & \varphi_{R} (\mu_{1}) f(\mu_{1}) - \lambda_{2} \varphi_{R} (\mu_{2}) f(\mu_{2})\| \leq |\lambda_{1}| \, | \varphi_{R}(\mu_{1}) - \varphi_{R} (\mu_{2}) | \, \| f(\mu_{1})\|
\leq 2 M_{f} ( 2 + R) W_{1} (\mu_{1}, \mu_{2})\,.
\end{align*}
In the third scenario, we have $m_{1}(\mu_{i}) >R+1$ for $i=1, 2$. Then, $\varphi_{R} (\mu_{i}) = 0$ for $i=1, 2$ and the inequality is trivial.

The map  $(\lambda, \mu) \mapsto ( 1 - \lambda) \varphi_{R} (\mu) e(\mu)$  in $[0, 1] \times \cP_{1} (\R^{d})$ can be treated in the very same way, relying on Remark~\ref{r:e-lip}. This concludes the proof of the proposition.
\end{proof}
\begin{remark}\label{r:RemarkOnSpaceMarginals}
For every $\Psi_1,\Psi_2 \in \cP_1(\R^d\times[0,1])$ we have
\begin{equation*}
    W_1(\mu_1,\mu_2) \le W_1(\Psi_1,\Psi_2) 
\end{equation*}
where $\mu_i = (\pi_x)_\sharp \Psi_i$.
Indeed, if $\gamma \in \Gamma(\Psi_1,\Psi_2)$ is a transport plan between $\Psi_1$ and $\Psi_2$\,, then $\tilde{\gamma} \coloneqq (\pi_x, \pi_x)_\sharp \gamma \in \Gamma(\mu_1,\mu_2)$, whence
\begin{equation*}
\begin{split}
    W_1 (\mu_1,\mu_2) &\le \int_{(\R^d)^2}\norm{x_1-x_2}\,\de\tilde{\gamma}(x_1,x_2) = \int_{(\R^d\times[0,1])^2}\norm{x_1-x_2} \,\de\gamma(x_1,\lambda_1,x_2,\lambda_2) \\
    &\le \int_{(\R^d\times[0,1])^2}(\norm{x_1-x_2} + \abs{\lambda_1 -\lambda_2})\,\de\gamma(x_1,\lambda_1,x_2,\lambda_2),
\end{split}
\end{equation*}
and the estimate follows from the arbitrariness of $\gamma \in \Gamma(\Psi_1,\Psi_2)$.
\end{remark}

For every $R>0$ we consider the \emph{truncated system}
\begin{equation}
\label{e:truncated-system}
    \begin{dcases}
        \de X^R_t = \nu\,\tv^R_{\rho^R_t}(X^R_t,\Lambda^R_t)\,\de t + \varsigma \norm{\tv^R_{\rho^R_t}(X^R_t,\Lambda^R_t)}\,\de B_t\,, \\
        \de \Lambda^R_t = \cT_{\Sigma^R_t}(X^R_t,\Lambda^R_t)\,\de t,\\
        (X^{R}_{0}, \Lambda^{R}_{0}) = (\widehat{X}_{0}, \widehat{\Lambda}_{0}),\\
          \Sigma^R_t = \Law(X^R_t,\Lambda^R_t), \\
    \rho^R_t = (\pi_x)_\sharp \Sigma^R_t = \Law(X^R_t)
    \end{dcases}
\end{equation}
in the time interval $[0, T]$. By virtue of Proposition~\ref{p-vR-lip} and Remark~\ref{r:RemarkOnSpaceMarginals}, the field $\tv^R$ satisfies condition ($w1$), which, combined with assumptions $(T1)$--$(T2)$, allows us to write that, for all $(\Psi_1,x_1,\lambda_1),(\Psi_2,x_2,\lambda_2)\in \cP_1(\R^d \times [0,1]) \times \R^d \times [0,1]$,
\begin{subequations}\label{e:VV}
\begin{align}
&\norm{\tv^R_{\mu_1}(x_1,\lambda_1) - \tv^R_{\mu_2}(x_2,\lambda_2)} \le L_R \big( \norm{x_1 - x_2} + \abs{\lambda_1 - \lambda_2}+W_1(\Psi_1,\Psi_2) \big),\\
&\abs{\cT_{\Psi_1}(x_1,\lambda_1) - \cT_{\Psi_2}(x_2,\lambda_2)} \le L_\cT \big( \norm{x_1 - x_2} + \abs{\lambda_1 - \lambda_2} + W_1(\Psi_1,\Psi_2) \big),
\end{align}
\end{subequations}
and, for every $(\Psi, x , \lambda)\in \cP_1(\R^d \times [0,1]) \times \R^d \times [0,1]$, 
\begin{equation}\label{eq:T_prop_geometry}
\lambda + \theta \cT_\Psi(x,\lambda) \in [0,1].
\end{equation}
Therefore, for every fixed $R > 0$, we are in a position to apply Proposition~\ref{p:existence-truncated}, ensuring the existence of a solution $Y^{R}= (X^{R}, \Lambda^{R})$ to~\eqref{e:truncated-system}, for any initial data  $\widehat{X}_0\colon\Omega\to\R^d$ $\sF_0$-measurable random variable belonging to $L^p(\Omega,\sF,\prob)$, $p \geq 2$, and $\widehat{\Lambda}_0\colon\Omega\to [0,1]$ $\sF_0$-measurable random variable, together with pathwise uniqueness in the class of strong solutions to~\eqref{e:truncated-system} such that the map $t \mapsto m_1(\rho^R_t)$ is bounded over $[0,T]$. In particular, $Y^{R}$ satisfies~\eqref{e:4-moment}.

In the next proposition, we provide a uniform control on~$m_{2} (\rho_{t}^{R})$ (and consequently on $m_{1} (\rho_{t}^{R})$) for $R>0$ and $t \in [0, T]$. This will allow us to prove the existence part of Theorem~\ref{t:well-posedness}.

\begin{proposition} 
\label{p:estimate-well-posed-truncated}
Under the assumptions of Theorem~\ref{t:well-posedness}, for $R>0$ let $Y^{R} = (X^{R}, \Lambda^{R})$ be the solution to~\eqref{e:truncated-system} in the interval $[0,T]$. Then, there exists a constant $A>0$ only depending on~$\nu$ and~$\varsigma$ such that
\begin{equation}
\label{e:estimate-well-posed-truncated}
    \sup_{t\in[0,T]}m_{2}^{2} (\rho^{R}_{t})  \le \bigl( \E( \| \widehat{X}_{0}\|^{2} )   + A\, T\bigr)e^{A\,  T},\quad \text{for every $R > 0$}.
\end{equation}
\end{proposition}

\begin{proof}
Let us apply It\^o formula to the process $Y^R$ and to the $C^2(\R^d)$ function $\psi(x) = \norm{x}^2$.
Therefore, for all~$t \in [0,T]$ we have  
\begin{equation*}
\begin{split}
\frac{\de}{\de t}\int_{\R^d \times [0,1]}\psi(x)\,\de\Sigma^R_t(x,\lambda) = &\ \nu \int_{\R^d \times [0,1]}\langle \nabla\psi(x), \tilde{v}^R_{\rho^R_t}(x,\lambda)\rangle\,\de\Sigma^R_t(x,\lambda) \\
&\ + \frac{1}{2}\int_{\R^d \times [0,1]}  \trace(\tsigma^R_{\rho^R_t}(x,\lambda)^T H_\psi(x)\tsigma^R_{\rho^R_t}(x,\lambda))\,\de\Sigma^R_t(x,\lambda),
\end{split}
\end{equation*}
whence 
\begin{equation*}
\begin{split}
\frac{\de}{\de t} m_{2}^{2} (\rho^{R}_{t})  =&\  2\nu \int_{\R^d \times [0,1]} \langle x, \tilde{v}^R_{\rho^R_t}(x,\lambda)\rangle\,\de\Sigma^R_t(x,\lambda) + \varsigma^2 d  \int_{\R^d \times [0,1]}  \norm{\tilde{v}^R_{\rho^R_t}(x,\lambda)}^2\,\de\Sigma^R_t(x,\lambda) \\
\eqqcolon &\  2\nu \, \mathrm{I}_t +  \varsigma^2 d \, \mathrm{II}_t\,.
\end{split}
\end{equation*}
We now estimate the two terms $\mathrm{I}_t, \mathrm{II}_t$ separately. By applying Cauchy–Schwarz inequality and recalling the expression of $\tv^R$, we obtain
\begin{align}
\label{e:It-1}
\mathrm{I}_t = &\int_{\R^d\times[0,1]} \langle x, \tilde{v}^R_{\rho^R_t}(x,\lambda)\rangle\,\de\Sigma^R_t(x,\lambda) \le \int_{\R^d\times[0,1]}\norm{x} \norm{\tilde{v}^R_{\rho^R_t}(x,\lambda)}\,\de\Sigma^R_t(x,\lambda) \\ 
\le& \int_{\R^d\times[0,1]}\norm{x} (\norm{x} + \norm{f(\rho^R_t)} + \norm{e(\rho^R_t)})\,\de\Sigma^R_t(x,\lambda) \nonumber\\
=& \int_{\R^d}\norm{x}^2\,\de\rho^R_t(x) + \int_{\R^d}\norm{x} \norm{f(\rho^R_t)}\,\de\rho^R_t(x) + \int_{\R^d}\norm{x} \norm{e(\rho^R_t)}\,\de\rho^R_t(x)\,.\nonumber
\end{align}
By $(f2)$, H\"older and Young inequality we infer from~\eqref{e:It-1} that
\begin{equation}
\label{e:It-2}
\begin{split}
\mathrm{I}_t & \leq m_{2}^{2} (\rho^{R}_{t}) + M_{f} ( 1 + m_{1} (\rho^{R}_{t}) ) m_{1} (\rho^{R}_{t}) + m_{1}^{2} (\rho^{R}_{t}) \leq 2 ( M_{f} + 1) m_{2}^{2} (\rho^{R}_{t})\,. 
\end{split}
\end{equation}
As for $\mathrm{II}_t$, in view of $(f2)$ we get that
\begin{align}
\label{e:IIt-1}
\mathrm{II}_t = &\int_{\R^d} \norm{\tilde{v}^R_{\rho^R_t}(x,\lambda)}^2\,\de\rho^R_t(x) \le \int_{\R^d} (\norm{x} + \norm{f(\rho^R_t)} + \norm{e(\rho^R_t)})^2\,\de\rho^R_t(x) \\ 
\le& \ 3\int_{\R^d}  (\norm{x}^2 + \norm{f(\rho^R_t)}^2 + \norm{e(\rho^R_t)}^2)\,\de\rho^R_t(x) 
= 
6\,  m_{2}^{2} (\rho^{R}_{t}) + 6\, M_{f}^{2} ( 1 + m_{2}^{2} (\rho^{R}_{t}) )\bigr). \nonumber
\end{align}

Therefore, combining~\eqref{e:It-2} and \eqref{e:IIt-1}  we conclude that for every $t \in [0,T]$
\begin{equation*}
\begin{split}
&m_{2}^{2}(\rho^{R}_{t}) \leq m_{2}^{2} (\rho^{R}_{0}) + A \int_{0}^{t} \bigl( 1 + m_{2}^{2} (\rho^{R}_{s})\bigr) \, \de s
\end{split} 
\end{equation*}
for some positive constant~$A$ only depending on~$\nu$ and~$\varsigma$. By Gr\"onwall inequality we infer that
\begin{displaymath}
 m_{2}^{2} (\rho^{R}_{t})  \le (m_{2}^{2} (\rho^{R}_{0})  + A t)e^{A t},\quad \text{for all } t \in [0,T]\,.
\end{displaymath}
Taking the supremum over $t \in [0, T]$ and recalling that $\rho^{R}_{0} = \Law (\widehat{X}_{0})$ we conclude for~\eqref{e:estimate-well-posed-truncated}.
\end{proof}

We are now in a position to prove Theorem~\ref{t:well-posedness}.

\begin{proof}[Proof of Theorem~\ref{t:well-posedness}]
We divide the proof in the existence and uniqueness parts.

\noindent{\bf Existence:} For all $R > 0$ we consider $Y^{R}= (X^{R}, \Lambda^{R})$ the solution to~\eqref{e:truncated-system} provided by Propositions~\ref{p:existence-truncated} and~\ref{p:estimate-well-posed-truncated}. We define the (deterministic) time
\begin{equation*}
    T_R \coloneqq \inf\bigl\{t \in [0,T] : \, m_{2} (\rho^{R}_{t})  > R\bigr\},
\end{equation*}
(we define $T_R = T$ when the set in empty). We observe that, as long as $t \le T_R$, the process $Y^R$ solves the original problem \eqref{e:system-SDE}, since the value of the truncation function $\varphi_R (\rho^{R}_{t})$ is constantly $1$. Moreover, by pathwise uniqueness of the solution to the truncated problem \eqref{e:truncated-system}, we have that, if $R_1 < R_2$\,, then $Y^{R_1} = Y^{R_2}$ on the time interval $[0, T_{R_1}]$, whence $T_{R_1} \le T_{R_2}$ (since the function $t \mapsto m_{2}
(\rho^{R}_{t}) $ is continuous over $[0,T]$, we have that $ m_{2}
(\rho^{R}_{T_{R}} )   = R$). Furthermore, by virtue of the uniform estimate~\eqref{e:estimate-well-posed-truncated}, if we choose~$R$ large enough to satisfy
\begin{equation*}
    \sup_{t\in[0,T]} m_{1} (\rho^{R}_{t}) \le \sup_{t\in[0,T]} m_2 (\rho^{R}_{t})  \le \sqrt{\bigl( \E( \| \widehat{X}_{0}\|^{2})   + A T\bigr)e^{A T}} \le R,
\end{equation*}
one has that $T_R = T$ and the truncated problem \eqref{e:truncated-system} coincide with the original one \eqref{e:system-SDE} on the entire time interval $[0,T]$. Therefore, with such a choice of $R$ we obtain that, for all $t \in [0,T]$,
\begin{equation*}
\begin{split}
& X^R_t = \widehat{X}_0 + \nu\int_0^t\tilde{v}^R_{\rho^R_s}(X^R_s,\Lambda^R_s)\,\de s + \varsigma\int_0^t \norm{\tilde{v}^R_{\rho^R_s}(X^R_s,\Lambda^R_s)}\,\de B_s \\
&\quad \,\,\, = \widehat{X}_0 + \nu\int_0^t v_{\rho^R_s}(X^R_s,\Lambda^R_s)\,\de s + \varsigma\int_0^t \norm{v_{\rho^R_s}(X^R_s,\Lambda^R_s)}\,\de B_s\,,\\
&\Lambda^R_t = \widehat{\Lambda}_0 + \int_0^t \cT_{\Sigma^R_s}(X^R_s,\Lambda^R_s)\,\de s.
\end{split}
\end{equation*}
Thus, $Y^R= (X^{R}, \Lambda^{R})$ is a strong solution to~\eqref{e:system-SDE}.
This concludes the existence part of the proof.
Estimate~\eqref{e:moment-estimate} follows from Proposition~\ref{p:estimate-well-posed-truncated} and from the compactness of $[0,1]$, whereas inequality~\eqref{e:momento4} is a consequence of the \emph{a priori} estimate of Proposition~\ref{p:a-priori-estimate-SDE}, which is valid thanks to Proposition~\eqref{p:some-estimates}, ensuring the validity of condition ($w2$) for the field $v$.


\noindent{\bf Uniqueness:}
As for uniqueness, we observe that, if we take two strong solutions $Y^1$, $Y^2$ to \eqref{e:system-SDE} such that the maps $t \mapsto m_1(\rho^i_t)$ ($i = 1,2$) are bounded over $[0,T]$, then they are also strong solutions of a particular instance of the $R$-truncated problem \eqref{e:truncated-system}, with $R$ large enough to ensure that
\begin{equation*}
    \sup_{t \in [0,T]} m_{1} (\rho^{i}_{t})  \le R, \quad i=1,2.
\end{equation*}
Since pathwise uniqueness (in this class of strong solutions) holds for the truncated problem \eqref{e:truncated-system} thanks to Proposition~\ref{p:existence-truncated}, it also hold for the original problem \eqref{e:system-SDE}.
\end{proof}

\begin{remark}
\label{r:aremark}
Notice that the dependence of the constant~$C$ in~\eqref{e:moment-estimate} on~$f$ is only due to the constant~$M_{f}\geq0$ defined in~$(f2)$. This can be observed by inspecting the proof of Proposition~\ref{p:estimate-well-posed-truncated} (cf.~\eqref{e:It-2}--\eqref{e:IIt-1}).
\end{remark}



\subsection{PDE formulation and finite-particle approximation}

We study the PDE counterpart of the SDE system~\eqref{e:system-SDE} and its approximation by means of a finite particles system.

\begin{proposition}
    \label{p:PDE-formulation-CBO}
    Given $T > 0$, let $B \colon \Omega \to C([0,T],\R^d)$ a $d$-dimensional standard Brownian motion defined on a filtered probability space $\bigl(\Omega, \sF, (\sF_t)_{t\in[0,T]},\prob\bigr)$, $\widehat{X}_0\colon\Omega\to\R^d$ an $\sF_0$-measurable random variable belonging to $L^p(\Omega,\sF,\prob)$ with $p \geq 2$, and $\widehat{\Lambda}_0\colon\Omega\to [0,1]$ an $\sF_0$-measurable random variable. Assume that $(f1)$--$(f2)$ and $(T1)$--$(T2)$ hold. Then, the PDE 
    \begin{equation}
    \label{e:PDE-CBO}
        \partial_{t} \Sigma_{t} - \frac{\varsigma^{2}}{2} \Delta_{x} ( \| v_{\rho_{t}} (x, \lambda) \|^{2} \Sigma_{t}) = {\rm div} \Big[ \big( \nu \, v_{\rho_{t}} (x, \lambda) \, , \,  \cT_{\Sigma_{t}} (x, \lambda)  \big) \Sigma_{t} \Big]\,,
    \end{equation}
    with initial datum $\Sigma_{0} = {\rm Law} (\widehat{X}_{0}, \widehat{\Lambda}_{0})$ and $\rho_{t} = (\pi_{x})_{\#} \Sigma_{t}$ admits a unique solution in $C([0, T]; \cP_{2} (\R^{d} \times [0, 1]))$.
\end{proposition}

\begin{proof}
The existence of solutions follows directly from Theorem~\ref{t:well-posedness}. Indeed, we show that the curve $t \mapsto \Sigma_{t}$ provided there is a solution to~\eqref{e:PDE-CBO}. To this end, let $\varphi \in C^{\infty}_{c} (\R^{d} \times \R)$ and let us test system~\eqref{e:system-SDE} using It\^o formula (cf.~\cite{Oksendal}):
\begin{align}
\label{e:test}
\varphi(X_{t}, \Lambda_{t}) = & \  \varphi(\widehat{X}_{0}, \widehat{\Lambda}_{0}) - \int_{0}^{t} \nu \nabla_{x} \varphi (X_{s}, \Lambda_{s}) \cdot v_{\rho_{s}} (X_{s}, \Lambda_{s}) \, \di s
\\
&
+ \int_{0}^{t} \nabla_{\lambda} \varphi (X_{s}, \Lambda_{s}) \, \cT_{\Sigma_{s}} (X_{s}, \Lambda_{s}) \, \di s  
+ \frac{\varsigma^{2}}{2} \int_{0}^{t} \Delta_{x} \varphi(X_{s}, \Lambda_{s}) \| v_{\rho_{s}} ( X_{s},  \Lambda_{s}) \|^{2} \, \di s \nonumber
\\
&
+ \varsigma \int_{0}^{t} \nabla_{x} \varphi(X_{s}) \| v_{\rho_{s}} (X_{s} , \Lambda_{s})  \|  \, \di B_{s}\,. \nonumber
\end{align}
Averaging~\eqref{e:test} over $\Om$ and using the fact that (see, e.g.,~\cite[Theorem~3.2.1]{Oksendal})
\begin{displaymath}
\EE \bigg( \int_{0}^{t} \nabla_{x} \varphi(X_{s}) \|v_{\rho_{s}} (X_{s} , \Lambda_{s})  \|  \di B_{s} \bigg) = 0\,,
\end{displaymath}
we deduce that
\begin{align*}
\label{e:test-2}
\int_{\R^{d} \times [0, 1]} \varphi(x, \lambda) \, \di \Sigma_{t} = & \ \int_{\R^{d}} \varphi(x, \lambda) \, \di \Sigma_{0} - \nu \int_{0}^{t} \int_{\R^{d} \times [0, 1]} \nabla_{x} \varphi(x, \lambda) \cdot v_{\rho_{s}} (x, \lambda )  \, \di \Sigma_{s} \, \di s 
\\
&
+\int_{0}^{t} \int_{\R^{d}\times[0, 1]} \nabla_{\lambda} \varphi(x, \lambda) \, \cT_{\Sigma_{s}} (x, \lambda) \, \di \Sigma_{s}  \, \di s  \nonumber
\\
&
+\frac{\varsigma^{2}}{2}  \int_{0}^{t} \int_{\R^{d} \times [0, 1]} \Delta_{x} \varphi(x, \lambda) \| v_{\rho_{s}} (x, \lambda)   \|^{2} \, \di \Sigma_{s} \, \di s \,.\nonumber
\end{align*}
Hence, $\Sigma_{t}$ solves~\eqref{e:PDE-CBO}. 

We now show that $\Sigma_{t}$ is the unique solution to~\eqref{e:PDE-CBO}. Let us assume that $t \mapsto \Theta_{t}$ is another solution to~\eqref{e:PDE-CBO}  in $C([0, T]; \cP_{2} (\R^{d} \times [0, 1]))$ and let $\theta_{t}\coloneqq (\pi_{x})_{\#} \Theta_{t}$. We notice that the linear PDE
\begin{equation}
\label{e:3.16}
\partial_{t} \Upsilon_{t} - \frac{\varsigma^{2}}{2} \Delta_{x} \big( \| v_{\theta_{t}} (x, \lambda) \|^{2} \Upsilon_{t} \big) = - {\rm div} \big(  (\nu \, v_{\theta_{t}} (x, \lambda) , \cT_{\Theta_{t}} (x, \lambda) \Upsilon_{t} )  \big)\,, \qquad \Upsilon_{0} = \Sigma_{0} 
\end{equation}
admits a unique solution in $C([0, T]; \cP_{2} (\R^{d} \times [0, 1]))$, which is indeed $\Theta_{t}$. This can be deduced using the Gr\"onwall argument of~\cite[Section 5.4]{Lions-LeBris}) with the choice $\sigma = \diag \big( \| v_{\theta_{t}} (x, \lambda) \|^{2} , \ldots,  \| v_{\theta_{t}} (x, \lambda) \|^{2} ,   0 \big) \in \R^{(d+1) \times (d+1)} $.

We further consider the auxiliary SDE system
\begin{align*}
\begin{cases}
\di X_{t} = \nu \, v_{\theta_{t}} (X_{t}, \Lambda_{t}) \, \di t + \varsigma \| v_{\theta_{t}} (X_{t}, \Lambda_{t}) \| \, \di B_{t}\,,\\
\di \Lambda_{t} = \cT_{\Theta_{t}} (X_{t}, \Lambda_{t}) \, \di t \,,\\
(X_{0}, \Lambda_{0}) = ( \widehat{X}_{0}, \widehat{\Lambda}_{0})\,.
\end{cases}
\end{align*} 
In view of (a simplified version of) Theorem~\ref{t:well-posedness}, the above system admits a unique solution $(\overline{X}_{t}, \overline{\Lambda}_{t})$, whose law $\Xi \coloneqq \Law ( \overline{X}, \overline{\Lambda})$ solves~\eqref{e:3.16}. This implies that $\Xi = \Theta$ and that $(\overline{X}_{t}, \overline{\Lambda}_{t})$ also solves~\eqref{e:system-SDE}. Thus, it must be $\Theta_{t} = \Sigma_{t}$ for every $t \in [0, T]$. 
\end{proof}

We now consider the finite particle approximation of~\eqref{e:system-SDE}. From now on, we work with $p =4$. Let $\widehat{X}_{0} \in L^4(\Omega,\sF,\prob)$ and $\widehat{\Lambda}^{i}_{0} \in L^{\infty} (\Om, \sF, \prob; [0, 1])$. For $N \in \mathbb{N}$, let $(\widehat{X}^{i}_{0}, \widehat{\Lambda}^{i}_{0})$ be $N$ random variables distributed as $(\widehat{X}_{0}, \widehat{\Lambda}_{0})$. We consider the SDE system
\begin{equation}
\label{e:particle-SDE}
\begin{cases}
\di X^{i, N}_{t} = \nu \, v_{\rho^{N}_{t}} (X^{i, N}_{t} , \Lambda^{i, N}_{t}) \, \di t + \varsigma \|  v_{\rho^{N}_{t}} (X^{i, N}_{t} , \Lambda^{i, N}_{t})\| \, \di B_{t}\,,\\[1mm]
\di \Lambda^{i, N}_{t} = \cT_{\Sigma^{N}_{t}} (X^{i, N}_{t} , \Lambda^{i, N}_{t}) \, \di t \,,\\[1mm]
(X^{i, N}_{0} , \Lambda^{i, N}_{0}) = ( \widehat{X}^{i}_{0}, \widehat{\Lambda}^{i}_{0})\,,\\[1mm]
\Sigma^{N}_{t} = \frac{1}{N} \sum_{j=1}^{N} \delta_{(X^{j, N}_{t} , \Lambda^{j, N}_{t})}\,, \ \rho^{N}_{t} =  \frac{1}{N} \sum_{j=1}^{N} \delta_{ X^{j, N}_{t} }\,.
\end{cases}
\end{equation}
In particular, $\rho^{N}_{t} = (\pi_{x})_{\#} \Sigma^{N}_{t}$. The first result concerns the well-posedness of~\eqref{e:particle-SDE}.

\begin{proposition}
\label{p:well-posed-finite}
 Given $T > 0$, let $B \colon \Omega \to C([0,T],\R^d)$ a $d$-dimensional standard Brownian motion defined on a filtered probability space $\bigl(\Omega, \sF, (\sF_t)_{t\in[0,T]},\prob\bigr)$, $\widehat{X}_0\colon\Omega\to\R^d$ an $\sF_0$-measurable random variable belonging to $L^4(\Omega,\sF,\prob)$, $\widehat{\Lambda}_0\colon\Omega\to [0,1]$ an $\sF_0$-measurable random variable, and let $(\widehat{X}^{i}_{0}, \widehat{\Lambda}^{i}_{0})$ be $N$ random variables distributed as $(\widehat{X}_{0}, \widehat{\Lambda}_{0})$. Assume that $(f1)$--$(f2)$ and $(T1)$--$(T2)$ hold. Then,~\eqref{e:particle-SDE} admits a unique solution~$(X^{i, N}_{t}, \Lambda^{i, N}_{t})_{i=1}^{N}$. Moreover, $(X^{i, N}_{t}, \Lambda^{i, N}_{t})$ are i.i.d.~and there exists $C>0$ such that for every $N \in \mathbb{N}$
 \begin{equation}
 \label{e:moments-finite}
 \sup_{\substack{j=1, \ldots, N\\ t \in [0, T]}} \E \Big( \| ( X^{j, N}_{t}, \Lambda^{j, N}_{t} ) \|^{2} + \| ( X^{j, N}_{t}, \Lambda^{j, N}_{t} ) \|^{4} \Big) \leq C \Big( 1 + \E ( \| \widehat{X}_{0} \|^{4}) \Big) \,.
 \end{equation}
\end{proposition}

\begin{proof}
Existence and uniqueness follow the truncation argument of Theorem~\ref{t:well-posedness}.\footnote{Notice that the truncation argument involves stochastic stopping times. Nonetheless this can be overcome using the techniques of~\cite[Section~9.5]{Baldi}.} The first estimate in~\eqref{e:moments-finite} can be obtained by testing equation~\eqref{e:particle-SDE} with $ \psi(x) = \| x\|^{2}$. Thanks to the integrability of the initial data we may recover~\eqref{e:moments-finite} applying It\^o formula and Gr\"onwall lemma. The solutions are i.i.d.~by the symmetry of the velocity fields.
\end{proof}

The compactness in law of $\Sigma^{N}_{t}$ and~$\rho^{N}_{t}$ follows by Aldous criterium~\cite[Section 16]{Billingsley}.

\begin{lemma}
\label{l:particle-compactness}
Under the assumptions of Proposition~\ref{p:well-posed-finite}, there exists a random measure $\Sigma \colon \Om \to \cP( C ([0, T];  \R^{d} \times [0, 1]))$ such that, up to a subsequence, ${\rm Law} (\Sigma^{N})$ converges narrow to ${\rm Law} (\Sigma)$ in $\cP(\cP (C([0, T]; \R^{d} \times [0, 1]) )) $. Furthermore, for every $t \in [0, T]$ the sequence ${\rm Law} (\Sigma_{t}^{N})$ converges to ${\rm Law}( \Sigma_{t})$ for every $t \in [0, T]$ in $(\cP_{2}( \cP_{2}(\R^{d} \times [0, 1])); W_{2})$.
\end{lemma}

\begin{proof}
Due to the symmetry of system~\eqref{e:particle-SDE}, it is enough to show that ${\rm Law} (X^{1, N}, \Lambda^{1, N})$ is tight in $\cP( C([0, T]; \R^{d}\times [0, 1]))$. To this end, we want to apply Aldous criterium (cf.~\cite[Section 16]{Billingsley}). We first show that ${\rm Law} (X^{1, N}_{t}, \Lambda^{1, N}_{t})$ is tight as a sequence in $\cP(\R^{d} \times [0, 1])$ for every $t \in [0, T]$. For $R>0$, let us estimate
\begin{align}
\label{e:9001}
\Law (X^{1, N}_{t}, \Lambda^{1, N}_{t}) ( (\R^{d} \setminus \overline{B}_{R}) \times [0, 1]) \leq \frac{ \E(\| (X^{1, N}_{t}, \Lambda^{1, N}_{t}) \| )}{R}\,.
\end{align}
In view of~\eqref{e:moments-finite}, $\Law (X^{1, N}_{t}, \Lambda^{1, N}_{t}) $ is tight in $\cP(\R^{d} \times [0, 1])$ for every $t \in [0, T]$.

As a second step, we fix $\delta_{0} >0$ and estimate for $\delta \in (0, \delta_{0})$ and $t \in [0, T-\delta]$
\begin{align*}
\E \big( \| (X^{1, N}_{t+\delta}, \Lambda^{1, N}_{t+\delta}) &  - (X^{1, N}_{t}, \Lambda^{1, N}_{t}) \|^{2} \big) 
\\
&
 \leq 2 \int_{t}^{t+\delta} \nu \,  \E( \|  v_{\rho^{N}_{s}} (X^{1, N}_{s}, \Lambda^{1, N}_{s})\| \, \| X^{1, N}_{s}\| ) + \E \big( | \cT_{\Sigma^{N}_{s}} (X^{1, N}_{s}, \Lambda^{1, N}_{s} ) | \, \Lambda^{1, N}_{s} \big) \, \di s 
\\
&
\quad + \varsigma^{2} d \int_{t}^{t+\delta}  \E \big(\| v_{\rho^{N}_{s}} ( X^{1, N}_{s}, \Lambda^{1, N}_{s})\|^{2} \big) \, \di s \,.
\end{align*}
By the assumptions~$(T1)$ and $(f2)$ we get that
\begin{align}
\label{e:translation estimate}
\E \big( \| (X^{1, N}_{t+\delta}, \Lambda^{1, N}_{t+\delta}) &   - (X^{1, N}_{t}, \Lambda^{1, N}_{t}) \|^{2} \big) 
\\
&
\leq C  \int_{t}^{t+\delta} \!\! \Big[ \E ( \| (X^{1, N}_{s}, \Lambda^{1, N}_{s}) \|^{2}) \Big]^{\frac{1}{2}} \bigg[1 + \Big( \E \Big( \sup_{j=1, \ldots, N}\, \| X^{1, N}_{s}\|^{2} \Big) \Big)^{\frac{1}{2}} \bigg] \di s  \nonumber
\end{align}
for some positive constant~$C$ independent of~$t$ and~$\delta$. In view of Proposition~\ref{p:well-posed-finite}, from \eqref{e:translation estimate} we infer that for every $\varepsilon, \eta >0$ there exists $\delta_{0}>0$ such that for every $t \in [0, T  - \delta_{0}]$
\begin{equation}
\label{e:9000}
\sup_{\delta \in (0, \delta_{0})} \, \mathbb{P} \Big( \|(X^{1, N}_{t+\delta}, \Lambda^{1, N}_{t+\delta})   - (X^{1, N}_{t}, \Lambda^{1, N}_{t}) \| \geq \eta \Big) < \varepsilon\,. 
\end{equation}
Inequalities~\eqref{e:9001} and \eqref{e:9000} yield the desired compactness of $\Law (\Sigma^{N})$. 

As for $\Law (\Sigma^{N}_{t})$, we notice that for any bounded and continuous map $\Gamma \colon \cP (\R^{d} \times [0, 1]) \to \R$, the map $\overline{\Gamma} \colon \cP (C([0, T]; \R^{d} \times [0, 1])) \to \R$ defined as $\overline{\Gamma} (\zeta) \coloneqq \Gamma(\zeta_{t})$ is bounded and continuous as well. Hence, we infer that  $\Law (\Sigma^{N}_{t})$ converges narrow to $\Law (\Sigma_{t})$ in $\cP (\cP( \R^{d} \times [0, 1]))$ for every $t \in [0, T]$. In view of the uniform bounds~\eqref{e:moments-finite}, we have that the sequence $\Law (\Sigma^{N}_{t})$ is also compact in $\cP_{2} (\cP_{2} (\R^{d} \times [0, 1]))$, since it holds
\begin{align*}
\int_{\cP_{2} (\R^{d} \times[0, 1])} \!\!\!\!\! \!\!\!\!\!\! W_{2}^{4} ( \zeta, \delta_{(0, 0)}) \, \di (\Sigma^{N}_{t})_{\#} \mathbb{P} (\zeta) & = \int_{\Om}  W_{2}^{4} ( \Sigma^{N}_{t}, \delta_{(0, 0)}) \, \di \mathbb{P}
\\
&
\leq  \int_{\Om} \frac{1}{N} \sum_{j=1}^{N} \| ( X^{j, N}_{t}, \Lambda^{j, N}_{t}) \|^{4} \, \di \mathbb{P} \leq \sup_{j=1, \ldots, N} \E \big( \| ( X^{j, N}_{t}, \Lambda^{j, N}_{t}) \|^{4} \big)\,. 
\end{align*}
Therefore, $\Law(\Sigma^{N}_{t})$ is also compact in $\cP_{2} (\cP_{2} (\R^{d}\times[0, 1]))$. This implies the last part of the thesis.
\end{proof}


For $N \in \mathbb{N}$ and $\varphi \in C^{2}_{b} (\R^{d} \times [0, 1])$ let us denote
\begin{align*}
G_{\varphi} (\Sigma^{N}) \coloneqq &  \int_{\R^{d} \times [0, 1] } \varphi (x, \lambda) \, \di \Sigma^{N}_{T} -  \int_{\R^{d} \times [0, 1]} \varphi(x, \lambda) \, \di \Sigma^{N}_{0} 
\\
&
- \int_{0}^{T} \int_{\R^{d} \times [0, 1] } \big( \nu \, v_{\rho^{N}_{t}} (x, \lambda), \cT_{\Sigma^{N}_{t}} (x, \lambda) \big)  \cdot \nabla \varphi (x, \lambda)  \, \di \Sigma^{N}_{t} \, \di t 
\\
&
+ \frac{\varsigma^{2}}{2} \int_{0}^{T}  \int_{\R^{d} \times [0, 1] } \| v_{\rho^{N}_{t}} (x, \lambda) \|^{2} \, \Delta_{x} \varphi_{x} (x, \lambda) \, \di \Sigma^{N}_{t} \, \di t \,.
\end{align*}
We will use a similar notation for $\Sigma$. The following holds.

\begin{lemma}
\label{l:Gvarphi}
Under the assumptions of Proposition~\ref{p:well-posed-finite}, for every $\varphi \in C^{2}_{b} (\R^{d} \times[0, 1])$ there exists a positive constant~$C = C( \| \nabla \varphi\|_{\infty}, \nu, \varsigma)$ such that for every $N \in \mathbb{N}$
\begin{equation}
\label{e:GvarphiN}
\E \big( | G_{\varphi} (\Sigma^{N}) |^{2} \big) \leq \frac{C}{N}\,. 
\end{equation}
\end{lemma}

\begin{proof}
By definition of~$G_{\varphi}$ and by It\^o formula, we have that
\begin{align*}
G_{\varphi} (\Sigma^{N}) = \int_{0}^{T} \frac{\varsigma}{N} \sum_{j=1}^{N} \| v_{\rho^{N}_{t}} (X^{j, N}_{t}, \Lambda^{j, N}_{t} ) \| \nabla_{x} \varphi (X^{j, N}_{t}, \Lambda^{j, N}_{t}) \, \di B_{t}\,.
\end{align*}
Taking the square and applying It\^o isometry (cf.~estimates~\eqref{e:moments-finite}), we have that
\begin{align*}
\E \big( | G_{\varphi} (\Sigma^{N})|^{2} \big) & = \frac{\varsigma^{2}}{N^{2}}  \E \bigg( \bigg| \int_{0}^{T} \sum_{j=1}^{N} \| v_{\rho^{N}_{t}} (X^{j, N}_{t}, \Lambda^{j, N}_{t} ) \| \nabla_{x} \varphi (X^{j, N}_{t}, \Lambda^{j, N}_{t}) \, \di B_{t} \bigg|^{2}\bigg) 
\\
&
= \frac{\varsigma^{2}}{N^{2}} \sum_{j=1}^{N} \E \bigg( \int_{0}^{T} \Bigg|  \| v_{\rho^{N}_{t}} (X^{j, N}_{t}, \Lambda^{j, N}_{t} ) \| \nabla_{x} \varphi (X^{j, N}_{t}, \Lambda^{j, N}_{t}) \Bigg|^{2} \, \di t \bigg)
\\
&
\leq  \frac{\varsigma^{2} \| \nabla \varphi\|^{2} }{N^{2}} \sum_{j=1}^{N}  \E \bigg( \int_{0}^{T}    \| v_{\rho^{N}_{t}} (X^{j, N}_{t}, \Lambda^{j, N}_{t} ) \|^{2}\, \di t \bigg) \leq \frac{C}{N}\,,
\end{align*}
where, in the last step, we have used the uniform bounds in Proposition~\ref{p:well-posed-finite}.
\end{proof}

We now prove that $t \mapsto \Sigma_{t}$ solves the PDE~\eqref{e:PDE-CBO} almost surely. First, we notice that by Lemma~\ref{l:particle-compactness} and by Skorokhod representation theorem~\cite[Section~6]{Billingsley}, up to redefining the probability space and the stochastic variable $\Sigma^{N}$, we may assume that $\Sigma^{N}$ converges narrowly to~$\Sigma$ in~$\cP(C([0, T]; \R^{d} \times [0, 1]))$ almost surely. A further application of Skorokhod representation theorem also allows us to assume that for every $t \in [0, T]$ the sequence $\Sigma^{N}_{t}$ converges to~$\Sigma_{t}$ in $(\cP_{2} (\R^{d} \times [0, 1]), W_{2})$ almost surely. Moreover, we still have that~\eqref{e:GvarphiN} holds. Finally, by~\eqref{e:moments-finite} we get
\begin{align}
\label{e:limit-moments}
\sup_{N \in \mathbb{N}} \, \sup_{t \in [0, T] } \, \E \bigg( m_{2}^{2} (\Sigma^{N}_{t}) + m_{4}^{4} (\Sigma^{N}_{t}) \bigg)  <+\infty\,.
\end{align}
By lower-semicontinuity, we infer from~\eqref{e:limit-moments} that
\begin{align}
\label{e:limit-moments-2} 
\sup_{t \in [0, T] } \, \E \bigg( m_{2}^{2} (\Sigma_{t}) + m_{4}^{4} (\Sigma_{t}) \bigg)  <+\infty\,.
\end{align}

\begin{proposition}
\label{p:limit-PDE-solved}
Under the assumptions of Proposition~\ref{p:well-posed-finite}, we have that for every $\varphi \in C^{2}_{b} ( \R^{d} \times [0, 1] ) $ it holds
\begin{displaymath}
\E ( | G_{\varphi} (\Sigma)| ) = 0 \,.
\end{displaymath}
In particular, the random variable $\Sigma$ is deterministic and is the unique solution to~\eqref{e:PDE-CBO} with initial condition~$\Sigma_{0}= \Law(\widehat{X}_{0}, \widehat{\Lambda}_{0})$.
\end{proposition}

\begin{proof}
In view of Lemma~\ref{l:Gvarphi}, by triangle inequality it is enough to estimate $\E ( | G_{\varphi} (\Sigma) - G_{\varphi} (\Sigma^{N})|)$. In particular, we have that
\begin{align}
\label{e:9006}
 \E  (  | & G_{\varphi}  (\Sigma)  - G_{\varphi} (\Sigma^{N})|)  \leq \E \bigg( \bigg|   \int_{\R^{d} \times [0, 1] } \varphi (x, \lambda) \, \di \Sigma_{T} -   \int_{\R^{d} \times [0, 1] } \varphi (x, \lambda) \, \di \Sigma^{N}_{T} \bigg| \bigg) 
\\
&
 + \E \bigg( \bigg| \int_{\R^{d} \times [0, 1]} \varphi(x, \lambda) \, \di \Sigma_{0} - \int_{\R^{d} \times [0, 1]} \varphi(x, \lambda) \, \di \Sigma^{N}_{0}  \bigg| \bigg) \nonumber
\\
&
  + \E \bigg( \bigg|  \int_{0}^{T} \int_{\R^{d} \times [0, 1] } \Big[ \big( \nu \, v_{\rho_{t}} (x, \lambda), \cT_{\Sigma_{t}} (x, \lambda) \big) -  \big( \nu \, v_{\rho^{N}_{t}} (x, \lambda), \cT_{\Sigma^{N}_{t}} (x, \lambda) \big) \Big] \cdot \nabla \varphi (x, \lambda)  \, \di \Sigma^{N}_{t} \, \di t \bigg| \bigg) \nonumber
\\
&
+ \E \bigg( \bigg|  \int_{0}^{T} \int_{\R^{d} \times [0, 1] } \big( \nu \, v_{\rho_{t}} (x, \lambda), \cT_{\Sigma_{t}} (x, \lambda) \big)  \cdot \nabla \varphi (x, \lambda)  \, \di ( \Sigma_{t} - \Sigma^{N}_{t} ) \, \di t  \bigg| \bigg) \nonumber
\\
&
+ \frac{\varsigma^{2}}{2} \E \bigg( \bigg| \int_{0}^{T}  \int_{\R^{d} \times [0, 1] } \Big[ \| v_{\rho_{t}} (x, \lambda) \|^{2} - \| v_{\rho^{N}_{t}} (x, \lambda) \|^{2}\Big] \Delta_{x} \varphi_{x} (x, \lambda) \, \di \Sigma^{N}_{t} \, \di t \bigg| \bigg) \nonumber
\\
&
+ \frac{\varsigma^{2}}{2} \E \bigg( \bigg| \int_{0}^{T}  \int_{\R^{d} \times [0, 1] }  \| v_{\rho_{t}} (x, \lambda) \|^{2} \, \Delta_{x} \varphi_{x} (x, \lambda) \, \di ( \Sigma_{t} - \Sigma^{N}_{t} )  \, \di t \bigg| \bigg)  \,.\nonumber
\end{align}
We pass to the limit term by term in the previous inequality.

By the regularity of $\varphi$, we estimate 
\begin{align*}
\E \bigg( \bigg|   \int_{\R^{d} \times [0, 1] } \varphi (x, \lambda) \, \di (\Sigma_{T} - \Sigma^{N}_{T})  \bigg| \bigg) \leq \| \nabla \varphi\|_{\infty} \E \big( W_{1} (\Sigma^{N}_{T}, \Sigma_{T}) \big) \,.
\end{align*}
Combining the almost sure convergence of~$\Sigma^{N}_{T}$ to~$\Sigma_{T}$ and the estimates~\eqref{e:limit-moments}--\eqref{e:limit-moments-2} we get that $\E \big( W_{1} (\Sigma^{N}_{T}, \Sigma_{T}) \big)\to 0$ as $N \to \infty$. Thus,
\begin{align}
\label{e:8001}
\lim_{N \to \infty} \E \bigg( \bigg|   \int_{\R^{d} \times [0, 1] } \varphi (x, \lambda) \, \di (\Sigma_{T} - \Sigma^{N}_{T})  \bigg| \bigg) = 0\,.
\end{align}

In view of the structure of~$v$, we have that
\begin{align}
\label{e:9002}
\E \bigg( \bigg| &   \int_{0}^{T} \int_{\R^{d} \times [0, 1] } \Big[  \nu \, v_{\rho_{t}} (x, \lambda) -   \nu \, v_{\rho^{N}_{t}} (x, \lambda) \Big] \cdot \nabla_{x} \varphi (x, \lambda)  \, \di \Sigma^{N}_{t} \, \di t \bigg| \bigg) 
\\
&
\leq \nu\, \E \bigg(  \int_{0}^{T} \bigg|   \int_{\R^{d} \times [0, 1] } \Big[  \lambda ( f(\rho_{t} ) - f( \rho^{N}_{t} ) ) + ( 1 - \lambda ) (e ( \rho_{t } ) - e ( \rho^{N}_{t} ) )  \Big] \cdot \nabla_{x} \varphi (x, \lambda)  \, \di \Sigma^{N}_{t} \bigg|   \di t  \bigg)\nonumber
\\
&
\leq  \int_{0}^{T} \E \bigg(   \bigg| ( f(\rho_{t} ) - f( \rho^{N}_{t} ) )  \cdot  \int_{\R^{d} \times [0, 1] }   \lambda  \nabla_{x} \varphi (x, \lambda)  \, \di \Sigma^{N}_{t} \bigg|   \bigg)  \di t  \nonumber
\\
&
\quad + \nu \int_{0}^{T} \E \bigg( \bigg|  (e ( \rho_{t } ) - e ( \rho^{N}_{t} ) )  \cdot  \int_{\R^{d} \times [0, 1]} ( 1 - \lambda )\nabla_{x} \varphi (x, \lambda)  \, \di \Sigma^{N}_{t} \bigg|   \bigg)  \di t  \,. \nonumber
\end{align}
Let us estimate the first term on the right-hand side of~\eqref{e:9002}. By H\"older inequality we write
\begin{align*}
 \int_{0}^{T} \E \bigg( &   \bigg| ( f(\rho_{t} ) - f( \rho^{N}_{t} ) )  \cdot  \int_{\R^{d} \times [0, 1] }   \lambda  \nabla_{x} \varphi (x, \lambda)  \, \di \Sigma^{N}_{t} \bigg|    \bigg)  \di t
 \\
 &
 \leq \int_{0}^{T} \bigg[ \E \bigg( \|  f(\rho_{t} ) - f( \rho^{N}_{t}) \|^{2} \bigg)\bigg]^{\frac{1}{2}}  \bigg[ \E \bigg( \bigg\|   \int_{\R^{d} \times [0, 1] }   \lambda  \nabla_{x} \varphi (x, \lambda)  \, \di \Sigma^{N}_{t} \bigg\|^{2}\bigg)  \bigg]^{\frac{1}{2}} \di t 
 \\
 &
 \leq \| \nabla \varphi\|_{\infty} \int_{0}^{T} \bigg[ \E \bigg( \|  f(\rho_{t} ) - f( \rho^{N}_{t}) \|^{2} \bigg)\bigg]^{\frac{1}{2}} 
\end{align*}
We recall that $ \|  f(\rho_{t} ) - f( \rho^{N}_{t}) \|^{2} \to 0$ almost surely. In view of $(f2)$ and of the control on $m_{2}^{4} (\rho_{t})$ and $m_{2}^{4} (\rho^{N}_{t})$ in~\eqref{e:limit-moments}--\eqref{e:limit-moments-2}, by Vitali convergence theorem and dominated convergence  we conclude that
\begin{displaymath}
\lim_{N \to \infty} \int_{0}^{T} \E \big( \|  f(\rho_{t} ) - f( \rho^{N}_{t}) \|^{2} \big) \, \di t =  0\, ,
\end{displaymath} 
so that
\begin{align}
\label{e:8002}
\lim_{N \to \infty} \E \bigg(    \int_{0}^{T} \bigg| ( f(\rho_{t} ) - f( \rho^{N}_{t} ) )  \cdot  \int_{\R^{d} \times [0, 1] }   \lambda  \nabla_{x} \varphi (x, \lambda)  \, \di \Sigma^{N}_{t} \bigg|   \di t  \bigg)  = 0\,.
\end{align}
In a similar way we obtain that
\begin{align}
\label{e:8003}
\lim_{N \to \infty}  \E \bigg( \int_{0}^{T} \bigg|  (e ( \rho_{t } ) - e ( \rho^{N}_{t} ) )  \cdot  \int_{\R^{d} \times [0, 1]} ( 1 - \lambda )\nabla_{x} \varphi (x, \lambda)  \, \di \Sigma^{N}_{t} \bigg|   \di t  \bigg) = 0\,.
\end{align}
Combining~\eqref{e:9002}--\eqref{e:8003} we conclude that
\begin{align}
\label{e:8004}
\lim_{N \to \infty} \E \bigg( \bigg| &   \int_{0}^{T} \int_{\R^{d} \times [0, 1] } \Big[  \nu \, v_{\rho_{t}} (x, \lambda) -   \nu \, v_{\rho^{N}_{t}} (x, \lambda) \Big] \cdot \nabla_{x} \varphi (x, \lambda)  \, \di \Sigma^{N}_{t} \, \di t \bigg| \bigg) = 0\,.
\end{align}

In view of~$(T1)$, we estimate
\begin{align*}
\E  \bigg( \bigg|   \int_{0}^{T} \int_{\R^{d} \times [0, 1] } & \big( \cT_{\Sigma_{t}} (x, \lambda)  -   \cT_{\Sigma^{N}_{t}} (x, \lambda) \big) \cdot \nabla \varphi (x, \lambda)  \, \di \Sigma^{N}_{t} \, \di t \bigg| \bigg) 
\\
&
\leq \| \nabla \varphi \|_{\infty} \int_{0}^{T} \E \bigg( \int_{\R^{d} \times [0, 1]} \big | \cT_{\Sigma_{t}} (x, \lambda)  -   \cT_{\Sigma^{N}_{t}} (x, \lambda) \big|   \, \di \Sigma^{N}_{t} \bigg) \di t \nonumber
\\
&
\leq  L_{\cT}  \| \nabla \varphi \|_{\infty} \int_{0}^{T} \E \big( W_{1} (\Sigma^{N}_{t}, \Sigma_{t}) \big) \, \di t  \,. \nonumber
\end{align*}
As $W_{1} (\Sigma^{N}_{t}, \Sigma_{t}) \to 0$ almost surely, and by~\eqref{e:limit-moments}--\eqref{e:limit-moments-2}, we have that $\E\big( W_{1} (\Sigma^{N}_{t}, \Sigma_{t}) \big) \to 0$ as $N \to \infty$ for every $t \in [0, T]$. Using again~\eqref{e:limit-moments}--\eqref{e:limit-moments-2}, we conclude that
\begin{displaymath}
\lim_{N \to \infty} \int_{0}^{T} \E \big( W_{1} (\Sigma^{N}_{t}, \Sigma_{t}) \big) \, \di t = 0\,.
\end{displaymath}
This in turn implies that
\begin{align}
\label{e:8005}
\lim_{N \to \infty} \E  \bigg( \bigg|   \int_{0}^{T} \int_{\R^{d} \times [0, 1] } & \big( \cT_{\Sigma_{t}} (x, \lambda)  -   \cT_{\Sigma^{N}_{t}} (x, \lambda) \big) \cdot \nabla \varphi (x, \lambda)  \, \di \Sigma^{N}_{t} \, \di t \bigg| \bigg)  = 0\,.
\end{align}

For fixed $t \in [0, T]$, we notice that the map $(x, \lambda) \mapsto ( v_{\rho_{t}} (x, \lambda) , \cT_{\Sigma_{t}} (x, \lambda)) \cdot \nabla \varphi (x, \lambda)$ belongs to $C_{b} (\R^{d} \times [0, 1])$. Hence, by construction of $\Sigma^{N}_{t}$ and~$\Sigma_{t}$ we have that for every $t \in [0, T]$
\begin{align*}
 \lim_{N \to \infty} \int_{\R^{d} \times [0, 1] } \big( \nu \, v_{\rho_{t}} (x, \lambda), \cT_{\Sigma_{t}} (x, \lambda) \big)  \cdot \nabla \varphi (x, \lambda)  \, \di ( \Sigma_{t} - \Sigma^{N}_{t} )  = 0 \qquad \text{almost surely.}
\end{align*}
Thanks to~\eqref{e:limit-moments}--\eqref{e:limit-moments-2}, we are again in a position to apply Vitali convergence theorem and dominated convergence to conclude that
\begin{align}
\label{e:8006}
\lim_{N \to \infty} &  \E \bigg( \bigg|  \int_{0}^{T} \int_{\R^{d} \times [0, 1] } \big( \nu \, v_{\rho_{t}} (x, \lambda), \cT_{\Sigma_{t}} (x, \lambda) \big)  \cdot \nabla \varphi (x, \lambda)  \, \di ( \Sigma_{t} - \Sigma^{N}_{t} ) \, \di t  \bigg| \bigg) 
\\
&
\leq \lim_{N \to \infty} \int_{0}^{T} \E \bigg( \bigg| \int_{\R^{d} \times [0, 1]}  \big( \nu \, v_{\rho_{t}} (x, \lambda), \cT_{\Sigma_{t}} (x, \lambda) \big)  \cdot \nabla \varphi (x, \lambda)  \, \di ( \Sigma_{t} - \Sigma^{N}_{t} ) \bigg| \bigg) \di t  = 0\,. \nonumber
\end{align}
With very similar arguments, we can also prove that
\begin{align}
\label{e:8007} & \lim_{N \to \infty} \frac{\varsigma^{2}}{2} \int_{0}^{T}  \E \bigg( \bigg|  \int_{\R^{d} \times [0, 1] } \Big[ \| v_{\rho_{t}} (x, \lambda) \|^{2} - \| v_{\rho^{N}_{t}} (x, \lambda) \|^{2}\Big] \Delta_{x} \varphi_{x} (x, \lambda) \, \di \Sigma^{N}_{t} \bigg| \bigg)  \di t  = 0\,,
\\
&
\label{e:8008} \lim_{N \to \infty} \int_{0}^{T}  \E \bigg(  \bigg| \int_{\R^{d} \times [0, 1] }  \| v_{\rho_{t}} (x, \lambda) \|^{2} \, \Delta_{x} \varphi_{x} (x, \lambda) \, \di ( \Sigma_{t} - \Sigma^{N}_{t} )\bigg| \bigg)  \di t  = 0 \,.
\end{align}
Therefore, passing to the limit in~\eqref{e:9006} and using~\eqref{e:8001}--\eqref{e:8008} we obtain that
\begin{displaymath}
\lim_{N \to \infty} \, \E( | G_{\varphi} (\Sigma ) - G_{\varphi} (\Sigma^{N} )| ) = 0\,.
\end{displaymath}
In particular, this implies that $\E(| G_{\varphi} (\Sigma) |) = 0$ for every $\varphi \in C^{2}_{b} (\R^{d} \times [0, 1])$ and $\Sigma$ solves~\eqref{e:PDE-CBO} almost surely. By uniqueness of solutions (cf.~Proposition~\ref{p:PDE-formulation-CBO}) we conclude that $\Sigma$ is deterministic. 
This concludes the proof of the proposition.
\end{proof}

\section{Long-time behaviour and concentration}
\label{s:auxiliary}

This section is devoted to the study of the long-time behaviour of system~\eqref{e:system-SDE} featuring a velocity field 
\begin{displaymath}
v^{n}_{\mu} (x, \lambda) \coloneqq -x + \lambda f_{n }(\mu) + (1 - \lambda ) e (\mu) \qquad \text{for $\mu \in \cP_{1} (\R^{d})$, $(x, \lambda) \in \R^{d} \times [0, 1]$,}
\end{displaymath}
where $f_{n} \colon \cP_{1} (\R^{d}) \to \R^{d}$ a sequence of functions satisfying~$(f1)$--$(f2)$ for every $n \in \mathbb{N}$. In particular, we allow the Lipschitz constant~$L_{f_{n}, R}$ to depend on $n \in \mathbb{N}$, while we assume the constant~$M_{f_{n}}\geq 0$ in $(f2)$ to be uniform in $n$, i.e., we can choose $M_{f_{n}} = M\geq 0$ for every $n \in \mathbb{N}$. We further assume that
\begin{itemize}
\label{e:uniform-conf-hp}
\item[$(f3)$] For every function $\ell \colon (0, +\infty) \to (0, 1)$, the sequence $f_{n}$ converges to $0$ uniformly on the set 
\begin{displaymath}
E_{\ell, R}\coloneqq \{ \gamma \in \cP_{2} (\R^{d}): \, m_{2} (\gamma) \leq R, \,  \gamma (B_{r}) \geq \ell(r)>0\}.
\end{displaymath}
\end{itemize}
We further need to strengthen the assumption on the field $\cT$ at $\lambda=0$:
\begin{itemize}
\item[$(T3)$] $\cT_{\Psi}(x,0)>0$ 
for every $\Psi \in \mathcal{P} (\R^{d} \times [0, 1])$ and every $x \in \R^{d}$;
\end{itemize}
notice that assumption $(T2)$ only provides the large inequality, and that $\cT_{\Psi}(x,1)\leq0$, for every $\Psi \in \mathcal{P} (\R^{d} \times [0, 1])$ and every $x \in \R^{d}$.

\smallskip

In Section~\ref{s:convergence} we study the concentration properties of the sequence of systems (from now on, we set, without loss of generality, $\nu=1$)
\begin{equation}
\label{e:system-n}
\begin{cases}
\di X_{t} = v^{n}_{\rho_{t}} (X_{t}, \Lambda_{t}) \, \di t +  \varsigma \| v^{n}_{\rho_{t}} (X_{t}, \Lambda_{t}) \| \di B_{t}\,,\\[1mm]
\di \Lambda_{t} = \cT_{\Sigma_{t}}(X_{t},\Lambda_{t})\, \di t\,, \\[1mm] 
 \Sigma_{t} = {\rm Law} (X_{t}, \Lambda_{t}) \\
         (X_{0}, \Lambda_{0}) = (\widehat{X}_{0}, \widehat{\Lambda}_{0}),\\
         \rho_{t} = (\pi_x)_\sharp \Sigma_{t} = \Law(X_t)\,.
\end{cases}
\end{equation}
where $B \colon \Omega \to C([0,T],\R^d)$ is a $d$-dimensional standard Brownian motion defined on a filtered probability space $\bigl(\Omega, \sF, (\sF_t)_{t\in[0,T]},\prob\bigr)$, $\widehat{X}_0\colon\Omega\to\R^d$ an $\sF_0$-measurable random variable belonging to $L^4(\Omega,\sF,\prob)$, and $\widehat{\Lambda}_0\colon\Omega\to [0,1]$ an $\sF_0$-measurable random variable. Our analysis relies on the study of the auxiliary system
\begin{equation}
\label{e:system2}
\begin{cases}
\di X_{t} = (- X_{t} + ( 1 - \Lambda_{t}) e(\mu_{t}) ) ) \di t + \varsigma \| X_{t} - ( 1 - \Lambda_{t}) e (\mu_{t} ) \| \di B_{t}\,,\\[1mm]
\di \Lambda_{t} = \cT_{\Psi_{t}}(X_{t},\Lambda_{t})\, \di t\,, \\[1mm] 
 \Psi_{t} = {\rm Law} (X_{t}, \Lambda_{t}) \\
         (X_{0}, \Lambda_{0}) = (\widehat{X}_{0}, \widehat{\Lambda}_{0}),\\
         \mu_{t} = (\pi_x)_\sharp \Psi_{t} = \Law(X_t)\,.
\end{cases}
\end{equation} 
which we carry out in Section~\ref{s:auxiliary-SDE}. Notice that \eqref{e:system2} is of the form~\eqref{e:system-SDE} with $f=0$ in the velocity~$v$. Hence, we refer to Theorem~\ref{t:well-posedness} for the well-posedness and estimates on the second moment of solutions to both~\eqref{e:system-n} and~\eqref{e:system2}, depending on $\EE(|\widehat{X}_{0}|^{2})$.  The PDE formulation of~\eqref{e:system2} follows from Corollary~\ref{p:PDE-formulation-CBO}.

\subsection{Concentration of the auxiliary system}  
\label{s:auxiliary-SDE}

We study the concentration properties as $t\to +\infty$ of the solution to the SDE system~\eqref{e:system2}. First, we focus on the space dynamics. Hence, we consider the SDE equation
\begin{equation}
\label{e:eq-sigma}
\di X_{t} = (- X_{t} + ( 1 - \Lambda_{t}) e(\mu_{t}) ) \di t + \varsigma \| X_{t} - ( 1 - \Lambda_{t}) e(\mu_{t})  \| \di B_{t}
\end{equation} 
with initial condition $X_{0} = \widehat{X}_{0} \in L^{4} (\Om,\sF,\prob)$, for a given random variable~$\Lambda\colon \Om \to  C([0, +\infty); [0, 1])$. The following long time behaviour holds. Again, we denote $\mu_{t} = \Law (X_{t})$.

\begin{theorem}
\label{t:concentration}
Let $B \colon \Omega \to C([0,+\infty),\R^d)$ be a $d$-dimensional standard Brownian motion defined on a filtered probability space $\bigl(\Omega, \sF, (\sF_t)_{t\in[0,+\infty)},\prob\bigr)$, $\widehat{X}_0\colon\Omega\to\R^d$ an $\sF_0$-measurable random variable belonging to $L^4(\Omega,\sF,\prob)$, and let $t \mapsto \Lambda_{t}$ be an $(\sF_t)_t$-adapted continuous stochastic process in $C\bigl([0,+\infty),  [0,1]\bigr)$. Assume that $\varsigma^{2} d <2$ and let $t \mapsto X_{t}$ denote the unique solution to~\eqref{e:eq-sigma} with initial condition~$X_{0} = \widehat{X}_{0}$. 
Then the following statements hold true:
\begin{enumerate}
\item there exists a constant $C(\varsigma,d)>0$ such that $\EE( \|X_t \|^2)\leq C(\varsigma,d) \EE( \| \widehat{X}_{0} \|^{2})$;

\item if $\Lambda_{t}$ satisfies
\begin{equation}
\label{e:assumption-lambda}
\lim_{t \to +\infty} \int_{0}^{t} \EE(\Lambda_{s}) \, \di s = +\infty\,,
\end{equation} 
then we have that
\begin{equation}
\label{e:concentration-3000}
\lim_{t\to +\infty} \EE( \|X_{t} \|^{2}) = 0\,.
\end{equation}
\end{enumerate}
\end{theorem}

\begin{proof}
By Theorem~\ref{t:well-posedness}, $X_{t}$ satisfies 
\begin{equation}
\label{e:ooo}
\EE \biggl[ \sup_{t \in [0, T]}  \|X_{t} \|^{4} \biggr]  <+\infty\qquad \text{for every $T \in [0, +\infty)$.}
\end{equation}

By taking expectations in \eqref{e:eq-sigma} and using \eqref{e:ooo}, we deduce the equation
\begin{equation}\label{e:eq-medie}
\frac{\de}{\de t}\EE(X_{t}) = -\EE(\Lambda_{t})\EE(X_{t})\qquad\text{with initial condition}\quad \EE(X_{0})=\EE(\widehat{X}_{0}),
\end{equation}
whose solution is 
\begin{equation}\label{e:eq-medie-sol}
\EE( X_{t} ) = \EE(\widehat{X}_{0}) \, \exp\bigg\{-\int_{0}^{t} \EE(\Lambda_{s})\,\de s\bigg\}.
\end{equation}
It is immediate to see that 
\begin{equation}\label{e:eq-medie-sol-bdd}
\| \EE(X_{t}) \| \leq \| \EE(\widehat{X}_{0}) \| \qquad\text{for every $t\in[0,+\infty)$.}
\end{equation}

We apply It\^o formula~\cite[Theorem~4.2.1]{Oksendal} to~\eqref{e:eq-sigma} with $\psi(x)  = \|x\|^{2}$, which belongs to $C^{2}(\R^{d})$, and we get, for every $t\in[0,+\infty)$ and for every $h>0$,
\begin{align}
\label{e:conc1}
\| X_{t+h} \|^{2} = & \  \| X_{t} \|^{2} - 2 \int_{t}^{t+h} \big[ \|X_{s} \|^{2} - ( 1 - \Lambda_{s}) X_{s} \cdot e (\mu_{s}) \big] \, \di s 
\\
&
+ \varsigma^{2} d \int_{t}^{t+h} \| X_{s} - (1 - \Lambda_{s}) e(\mu_{s}) \|^{2} \, \di s 
+2 \varsigma  \int_{t}^{t+h} \big[ \| X_{s} \|^{2} - ( 1 - \Lambda_{s}) X_{s} \cdot e(\mu_{s}) \big] \, \di B_{s}\,. \nonumber
\end{align}

Averaging~\eqref{e:conc1} over~$\Om$, thanks to~\eqref{e:ooo} we infer that
\begin{align}
\label{e:conc2}
\EE( \|X_{t+h} \|^{2})  =&\,  \EE( \|{X}_{t} \|^{2} ) - 2\int_{t}^{t+h} \big[ \EE( \|X_{s} \|^{2} ) - \EE \big( ( 1 - \Lambda_{s}) X_{s}\big) \cdot \EE(X_{s}) \big] \, \di s 
\\
&
\quad + \varsigma^{2} d \int_{t}^{t+h} \EE ( \| X_{s} - (1 - \Lambda_{s}) \EE(X_{s}) \|^{2} )\, \di s  \nonumber
\\
=&\,
\EE( \|{X}_{t} \|^{2} ) - 2\int_{t}^{t+h} \big[ \EE( \|X_{s} \|^{2} ) - \EE \big( ( 1 - \Lambda_{s}) X_{s}\big) \cdot \EE(X_{s}) \big] \, \di s  \nonumber
\\
&
\quad+ \varsigma^{2} d \int_{t}^{t+h} \big[ \EE ( \| X_{s} \|^{2}) + \EE \big( (1 - \Lambda_{s})^{2}\big)  \|\EE(X_{s}) \|^{2} \big] \, \di s \nonumber
\\
&
\quad -2 \varsigma^{2} d \int_{t}^{t+h}  \EE \big(  (1 - \Lambda_{s}) X_{s} \big) \cdot \EE(X_{s})  \, \di s  \nonumber
\\
=&\,
\EE( \|{X}_{t} \|^{2} ) - 2 \bigg(1 - \frac{\varsigma^{2} d}{2} \bigg) \int_{t}^{t+h} \EE( \|X_{s} \|^{2})\, \di s \nonumber
\\
&
\quad +  2 (1 - \varsigma^{2} d ) \int_{t}^{t+h}  \EE \big( (1 - \Lambda_{s}) X_{s}\big) \cdot \EE(X_{s}) \, \di s\nonumber
\\
&
\quad + \varsigma^{2} d \int_{t}^{t+h} \EE \big(  (1 - \Lambda_{s})^{2} \big) \|\EE(X_{s}) \|^{2}\, \di s \nonumber\,.
\end{align}
We estimate the second to last term on the right-hand side of~\eqref{e:conc2}. By using \eqref{e:eq-medie-sol-bdd}, Jensen's and weighted Young's inequalities, and the fact that $\Lambda_{s} \in [0, 1]$, we get that for every $\eta>0$
\begin{align*}
\bigg\| \int_{t}^{t+h}  \EE \big( (1 - \Lambda_{s}) X_{s}\big) \cdot \EE(X_{s}) \, \di s \bigg \| & \leq   \int_{t}^{t+h}  \big\| \EE \big( (1 - \Lambda_{s}) X_{s} \big)\big\| \, \| \EE(X_{s}) \|  \, \di s \leq
\int_{t}^{t+h} \EE( \| X_{s} \|) \,  \big\| \EE(\overline{X}_{0}) \big\| \, \di s 
\\
&
\leq
\int_{t}^{t+h}  \bigg[\frac{\eta}{2} \, \EE(\|X_{s} \|^2) + \frac{1}{2\eta}\,  \big\| \EE(\overline{X}_{0}) \big\|^2\bigg] \, \di s\,.
\end{align*}
By choosing $\eta= (1-\frac{\varsigma^{2} d}{2} ) /2|1- \varsigma^{2} d|$, we can  continue in \eqref{e:conc2} with
\begin{align}
\label{e:conc100}
\EE( \|X_{t+h}\|^{2})  
\leq&\,
\EE( \|{X}_{t} \|^{2} ) - \bigg( 1  -  \frac{\varsigma^{2} d}{2} \bigg) \int_{t}^{t+h} \EE( \|X_{s} \|^{2})\, \di s + h \bigg(\frac{|1 -  \varsigma^{2} d | }{1- \frac{\varsigma^{2} d}{2} } + \varsigma^{2} d\bigg) \big \|\EE(\overline{X}_{0})\big\|^{2}\,. \nonumber
\end{align}
By dividing by $h$ and taking the limit as $h\to0^+$, we deduce that 
\begin{equation}\label{e:dividing-by-h}
\frac{\de}{\de t}\EE( \|X_{t} \|^{2}) \leq - \bigg( 1  -  \frac{\varsigma^{2} d}{2} \bigg) \EE(\| X_{t} \|^{2}) + \bigg(\frac{|1 -  \varsigma^{2} d| }{1- \frac{\varsigma^{2} d}{2}} + \varsigma^{2} d\bigg) \big\| \EE ( \widehat{X}_{0} )\big\|^{2}.
\end{equation}
By Gr\"onwall's Lemma, we obtain the estimate
\begin{equation}\label{e:estimate}
\begin{split}
\EE( \|X_{t} \|^{2}) \leq&\, \EE( \|\widehat{X}_{0} \|^{2}) \bigg[e^{-(1- \frac{\varsigma^{2} d}{2})t} + \frac{ 4 |1- \varsigma^{2} d |+ 2\varsigma^{2} d( 2 -\varsigma^{2} d ) }{( 2 -\varsigma^{2} d )^2}\big(1-e^{-(1- \frac{\varsigma^{2} d}{2} )t}\big)\bigg] 
\\
\leq&\, \EE( \|\widehat{X}_{0} \|^{2}) \bigg[1+ \frac{ 4 |1- \varsigma^{2} d |+ 2\varsigma^{2} d( 2 -\varsigma^{2} d ) }{( 2 -\varsigma^{2} d )^2} \bigg] \eqqcolon C(\varsigma,d) \EE( \|\widehat{X}_{0} \|^{2}),
\end{split}
\end{equation}
for every $t\in[0,+\infty)$, which concludes the proof of (1).

To prove (2), we first notice that \eqref{e:assumption-lambda} and \eqref{e:eq-medie-sol} imply that 
$ \|\EE(X_{t}) \|\to0$ as $t\to+\infty$, so that, for every $\eps>0$ there exists $t_0\in[0,+\infty)$ such that 
\begin{equation}\label{e:controllo}
\|\EE(X_{t}) \| < \eps\qquad\text{for every $t>t_0$\,.}
\end{equation}
We can now estimate of the last two integrals in \eqref{e:conc2}, for $t>t_0$ and $\eps\in(0,1)$, by
\begin{align}\label{e:better-estimates}
2(1 - \varsigma^{2} d) & \int_{t}^{t+h}  \EE \big( (1 - \Lambda_{s}) X_{s}\big) \cdot \EE(X_{s}) \, \di s  + \varsigma^{2} d \int_{t}^{t+h} \EE \big(  (1 - \Lambda_{s})^{2} \big) \| \EE(X_{s}) \|^{2}\, \di s 
\\
&
\leq
2|1-\varsigma^{2} d| \int_{t}^{t+h} \big \| \EE \big( (1 - \Lambda_{s}) X_{s}\big) \big\| \, \| \EE(X_{s}) \| \,\de s + \varsigma^{2} d \eps h  \nonumber
\\
&
\leq 2|1- \varsigma^{2} d| \eps \int_{t}^{t+h} \EE(\| X_{s} \| )\,\de s + \varsigma^{2} d\eps h \nonumber
\\
&
\leq  2|1- \varsigma^{2} d| \eps \int_{t}^{t+h} \big[\EE( \| X_{s} \|^2)\big]^{1/2}\,\de s + \varsigma^{2} d\eps h  \nonumber
\\
& 
\leq 2\eps h\Big( |1-\varsigma^{2} d| \sqrt{C(\varsigma ,d) \EE( \|\widehat{X}_{0} \|^{2})} + \varsigma^{2} d\Big). \nonumber
\end{align}
By combining \eqref{e:conc2} and \eqref{e:better-estimates}, we obtain
\begin{equation}\label{e:derivata2}
\frac{\de}{\de t} \EE( \| X_t \|^2) \leq - (2 - \varsigma^{2} d )  \EE( \| X_t \|^2) + 2\eps \Big( |1-\varsigma^{2} d| \sqrt{C(\varsigma,d) \EE( \| \widehat{X}_{0} \|^{2})} + \varsigma^{2} d\Big), 
\end{equation}
from which, applying Gronwall's Lemma once more, we get
\begin{equation}\label{e:Gronwall-once-more}
\EE( \| X_{t} \|^2) \leq \EE( \| \widehat{X}_{0} \|^{2}) e^{-( 2 - \varsigma^{2} d)t} + 2 \eps\frac{|1- \varsigma^{2} d | \sqrt{C(\varsigma ,d) \EE( \| \widehat{X}_{0} \|^{2})} + \varsigma^{2} d}{ 2 -\varsigma^{2} d} \big(1-e^{-( 2 - \varsigma^{2} d)t}\big),
\end{equation}
for every $t>t_0$.
By taking the $\limsup$ as $t\to+\infty$, we have
$$\limsup_{t\to+\infty} \EE( \| X_{t} \|^2) \leq 2 \eps \, \frac{| 1- \varsigma^{2}  d | \sqrt{C(\varsigma,d) \EE( \| \widehat{X}_{0} \|^{2})} + \varsigma^{2} d}{2 - \varsigma^{2} d}, $$
which yields \eqref{e:concentration-3000}, since $\eps$ is arbitrary. This concludes the proof.
%
%
\end{proof}

In the remaining part of this section, we analyze the complete SDE system~\eqref{e:system2}. The main concentration result is contained in the next theorem.


\begin{theorem}
\label{t:concentration2}
Assume that  $(T1)$--$(T3)$ hold. Let $B \colon \Omega \to C([0,+\infty),\R^d)$ be a $d$-dimensional standard Brownian motion defined on a filtered probability space $\bigl(\Omega, \sF, (\sF_t)_{t\in[0,+\infty)},\prob\bigr)$, $\widehat{X}_0\colon\Omega\to\R^d$ an $\sF_0$-measurable random variable belonging to $L^4(\Omega,\sF,\prob)$, and $\widehat{\Lambda}_0\colon\Omega\to [0,1]$ an $\sF_0$-measurable random variable belonging to $L^1(\Omega,\sF,\prob)$. Let $Y = (X, \Lambda)\colon \Omega \to C\bigl([0,T],\R^d\times [0,1]\bigr)$ be the $(\sF_t)_t$-adapted continuous stochastic process solution to~\eqref{e:system2}. If $\varsigma^{2} d < 2$, then 
\begin{equation}
\label{e:concentration2}
\lim_{t\to + \infty} \EE( \|X_{t} \|^{2}) = 0\,.
\end{equation}
\end{theorem}

\begin{proof}
Under the above assumptions, we have shown in Theorem~\ref{t:concentration} that $t \mapsto \EE( \| X_{2} \|^{2})$ is bounded, using only the fact that $\Lambda_{t} \in [0, 1]$, namely, 
\begin{equation}
\label{e:first-inequality-mean}
\EE( \| X_{t} \|^{2}) \leq C(\varsigma ,d)\EE(\| \widehat{X}_{0} \|^{2}) <+\infty \qquad \text{for every $t \in [0, +\infty)$.}
\end{equation}
In particular, setting $m_2(\Sigma) \coloneqq \int_{\R^{d}\times[0,1]} (|x|^2+\lambda^2)\,\de \Sigma(x,\lambda)$, for $\Sigma\in \cP(\R^{d}\times[0,1])$, we deduce that 
\begin{equation}\label{e:gli-possiamo-dare-un-nome}
m_2(\Psi_{t}) \leq C(\varsigma,d)\EE( \|\widehat{X}_0 \|^2)+1 \eqqcolon R.
\end{equation}

By the definition of~$\Psi_{t}$ and by Chebyschev inequality, for every $C>0$ and every $t \in [0, +\infty)$ it holds that
\begin{align}
\label{e:cheb}
\Psi_{t} (\{ (x, \lambda) \in \R^{d} \times [0, 1]: \| x \| >C\}) &  =  \mathbb{P} ( \{ \omega \in \Om: \, \| X_{t} \| >C\})
\\
&
\nonumber
\leq
\frac{\EE( \| X_{t} \|^{2})}{C^{2}} \leq \frac{C(\varsigma,d)\EE(\| \widehat{X}_{0} \|^{2})}{C^{2}}\,.
\end{align}
Choosing $C>0$ so that $ \frac{C(\varsigma,d)\EE( \| \widehat{X}_{0} \|^{2})}{C^{2}}<\frac{1}{2}$, \eqref{e:cheb} implies that for every $t \in [0, +\infty)$ we have
\begin{align}
\label{e:cheb2}
\Psi_{t} (\{ (x, \lambda) \in \R^{d} \times [0, 1]: \| x \| \leq C\}) > \frac{1}{2}\,.
\end{align}
Taking the average~$\EE(\Lambda_{t})$ in~\eqref{e:system2} and using, in order, assumptions~$(T1)$ and $(T3)$, \eqref{e:cheb2}, and \eqref{e:gli-possiamo-dare-un-nome}, for $t \in (0, +\infty)$ we have that 
\begin{align}\label{e:2.44}
\frac{\di}{\di t}  \EE(\Lambda_{t}) &  =  \EE(\cT_{\Psi_{t}}(X_{t},\Lambda_{t})) = \int_{\Omega} \cT_{\Psi_{t}}(X_{t},\Lambda_{t}) \, \di \mathbb{P}(\omega) = \int_{\R^{d} \times [0, 1]} \cT_{\Psi_{t}}(x,\lambda)  \, \di \Psi_{t} (x, \lambda)
\\
&=
\int_{\R^{d} \times [0, 1]} (\cT_{\Psi_{t}} (x,\lambda) - \cT_{\Psi_{t}}(x,0))\, \di \Psi_{t} (x, \lambda) + \int_{\R^{d} \times [0, 1]}  \cT_{\Psi_{t}}(x,0)  \, \di \Psi_{t} (x, \lambda) \nonumber
\\
&\geq
- L_{\cT} \int_{\R^{d}\times[0,1]} \lambda\, \de\Psi_{t}(x,\lambda) + \int_{\{(x,\lambda)\in\R^{d}\times[0,1]: \| x \|\leq C\}} \cT_{\Psi_{t}}(x,0)  \, \di \Psi_{t} (x, \lambda) \nonumber
\\
&\geq
- L_{\cT}\,  \EE(\Lambda_{t}) + \frac{1}{2} \inf \big\{\cT_{\Psi_{t}}(x,0): \| x \| \leq C, t\in[0,+\infty) \big\} \,, \nonumber 
\\
&\geq
- L_{\cT}\,  \EE(\Lambda_{t}) + \frac{1}{2} \inf \big\{\cT_{\Sigma}(x,0): \| x \| \leq C, \Sigma\in\cP_2(\R^{d}\times[0,1]) \text{ with } m_2(\Sigma)\leq R \big\} \,. \nonumber 
\end{align}
We notice that the Lipschitz assumption $(T1)$ implies that 
$$ I\coloneqq \inf \big\{\cT_{\Sigma}(x,0): \| x \| \leq C, \Sigma\in\cP_2(\R^{d}\times[0,1]) \text{ with } m_2(\Sigma)\leq R \big\} >0$$
as the set $\{\| x \| \leq C, \Sigma\in\cP_2(\R^{d}\times[0,1]) \text{ with } m_2(\Sigma)\leq R\}$ is compact in $\R^{d}\times \cP_1(\R^{d}\times[0,1])$.

Hence, we may fix $\delta > 0$ such that $\delta < \min \{ \EE(\widehat{\Lambda}_{0}), \frac{I}{2L_{\cT}}\}$.
We claim that for $t \in [0, +\infty)$ it must be $\EE(\Lambda_{t})> \delta$. By contradiction, assume that such condition is not satisfied. Then, we define $\overline{t} \in (0,+\infty)$ as
\begin{displaymath}
\overline{t}:= {\rm argmin}\{ t \in [0, +\infty): \, \EE(\Lambda_{t}) \leq \delta\}\,.
\end{displaymath}
We notice that it must be $\frac{\di}{\di t} \EE(\Lambda_{\overline{t}}) \leq 0$ and $\EE(\Lambda_{\overline{t}}) = {\delta}$. By~\eqref{e:2.44} and by our choice of~$\delta$ we have that
\begin{displaymath}
\frac{\di}{\di t} \EE(\Lambda_{\overline{t}}) \geq -L_{\cT} \, \delta+\frac{I}{2}>0\,.
\end{displaymath}
This is in contradiction with the condition $\frac{\di}{\di t} \EE(\Lambda_{\overline{t}}) \leq 0$. Hence, $\EE(\Lambda_{t}) \geq \delta$ for every $t \in [0, +\infty)$ and, as a consequence,
\begin{equation}
\label{e:div-condition}
\lim_{t\to+\infty} \int_{0}^{t} \EE(\Lambda_{t}) \, \di t = +\infty\,.
\end{equation}
Theorem~\ref{t:concentration}(2) yields \eqref{e:concentration2}.
\end{proof}

\subsection{Convergence as $n \to \infty$}
\label{s:convergence}

Relying on the results of Section~\ref{s:auxiliary-SDE}, we now study the long-time behaviour of~\eqref{e:system-n}. 
For $n \in \mathbb{N}$, we denote by $(X^{n}_{t}, \Lambda^{n}_{t})$ the solutions to~\eqref{e:system-n} in $[0, +\infty)$. For $r>0$ we introduce the function $\alpha_{r}\in C^{\infty}_{c}( [0, +\infty))$ as
\begin{displaymath}
\alpha_{r} (t) \coloneqq\begin{cases}
e^{1 - \frac{r^{2}}{r^{2} - t^{2}}} & \text{for $t \leq r$}\,,\\
0 & \text{otherwise}\,.
\end{cases}
\end{displaymath}
We further let $\phi_{r} (x) \coloneqq \alpha_{r} ( \| x\|)$ for $x \in \R^{d}$. The next technical lemma provides a useful estimate for $\phi_{r}$\,.

\begin{lemma}
\label{l:phi-R}
For every $R, r >0$ there exists $q = q(R, r) >0$ such that for every $w \in \overline{B}_{R}$, and every $x \in \R^{d}$
\begin{equation}
\label{e:ineq-phi-r}
\nabla \phi_{r} (x) \cdot (w - x) + \frac{\varsigma^{2}}{2} \Delta \phi_{r} (x) \| w - x \|^{2} \geq -q \phi_{r} (x) \,.
\end{equation}
\end{lemma}

\begin{proof}
We represent $\nabla \phi_{r}$ and $\Delta \phi_{r}$ in polar coordinates, in the point $x = \rho e_{\rho}$, where $\rho = \|x \|$ and $e_{\rho}$ denotes the radial unit vector. Then, by a direct computation we have that for $\rho \in [0, r)$
\begin{align*}
\nabla \phi_{r} (x) & = \alpha_{r}'(\rho) \frac{x}{\rho} = -\frac{2 \rho r^{2}\alpha_{r} (\rho)}{\rho (r^{2} - \rho^{2})^{2}}\, x\,,\\
\Delta \phi_{r} (x) & = \alpha''_{r} (\rho) + (d-1) \frac{\alpha'_{r} (\rho)}{\rho} 
= \frac{2 r^{2}\alpha_{r} (\rho) }{(r^{2} - \rho^{2})^{2}} \bigg( \frac{2\rho^{2}r^{2}}{(r^{2} - \rho^{2})^{2}}  - \frac{4\rho^{2}}{(r^{2} - \rho^{2})} - d \bigg) \,.
\end{align*} 
We now consider a vector $w \in \R^d$ with $\| w \| \leq R$ and such that $w= (w \cdot e_{\rho}) e_{\rho}$. We rewrite $w = w - x + x = x + t e_{\rho} + w^{\bot}$, for some $t \in [- R, R]$ and $w^{\bot} \perp e_{\rho}$ with $\| w^{\bot} \| \leq R$. Then,
\begin{align}
\label{e:formula-nabla-delta}
\nabla & \phi_{r} (x) \cdot (w - x) + \frac{\varsigma^{2}}{2}  \Delta \phi_{r} (x) \| w - x \|^{2} 
\\
&=  -\frac{2 \rho r^{2}\alpha_{r} (\rho)}{ (r^{2} - \rho^{2})^{2}}\, t + ( |t|^{2} + \| w^{\bot} \|^{2})  \frac{ \varsigma^{2}  r^{2}\alpha_{r} (\rho) }{(r^{2} - \rho^{2})^{2}} \bigg( \frac{2\rho^{2}r^{2}}{(r^{2} - \rho^{2})^{2}}  - \frac{4\rho^{2}}{(r^{2} - \rho^{2})} - d \bigg) \nonumber
\\
&
= 2r^{2}\alpha_{r} (\rho)  \bigg[ - \frac{\rho t }{(r^{2}- \rho^{2})^{2}} + \frac{\varsigma^{2}}{2}  (  t^{2} + \| w^{\bot} \|^{2}) \bigg( \frac{2\rho^{2}r^{2}}{(r^{2} - \rho^{2})^{4}}  - \frac{4\rho^{2}}{(r^{2} - \rho^{2})^{3}} - \frac{d}{(r^{2}- \rho^{2})^{2}} \bigg)  \bigg] \,. \nonumber
\end{align}
We notice that there exists $c_{1} (R,  r) < 0$, independent of $w$, such that
\begin{displaymath}
\| w^{\bot} \|^{2} \bigg( \frac{2\rho^{2}r^{2}}{(r^{2} - \rho^{2})^{4}}  - \frac{4\rho^{2}}{(r^{2} - \rho^{2})^{3}} - \frac{d}{(r^{2}- \rho^{2})^{2}} \bigg) \geq c_{1} (R, r)\,.
\end{displaymath}
This is a consequence of the continuity with respect to $\rho$ and of the limits
\begin{align*}
\lim_{\rho \to r} \frac{2\rho^{2}r^{2}}{(r^{2} - \rho^{2})^{4}}  - \frac{4\rho^{2}}{(r^{2} - \rho^{2})^{3}} - \frac{d}{(r^{2}- \rho^{2})^{2}}  & = + \infty \,,\\
\lim_{\rho \to 0} \frac{2\rho^{2}r^{2}}{(r^{2} - \rho^{2})^{4}}  - \frac{4\rho^{2}}{(r^{2} - \rho^{2})^{3}} - \frac{d}{(r^{2}- \rho^{2})^{2}} &  = -\frac{d}{r^{2}} \,.
\end{align*}
It is therefore enough to control the terms multiplying~$t$ in~\eqref{e:formula-nabla-delta}. Notice that we may always assume $t \neq 0$ in what follows.

For $\rho \in [0, \frac{r}{\sqrt{2}}]$ we have 
\begin{align*}
-\frac{\rho t}{(r^{2}- \rho^{2})^{2}} &  + \frac{\varsigma^{2}}{2}  t^{2} \bigg( \frac{2\rho^{2}r^{2}}{(r^{2} - \rho^{2})^{4}}  - \frac{4\rho^{2}}{(r^{2} - \rho^{2})^{3}} - \frac{d}{(r^{2}- \rho^{2})^{2}}\bigg) 
\\
&
 \geq -\frac{ \rho | t |}{(r^{2}- \rho^{2})^{2}}  + \frac{\varsigma^{2}}{2}  t^{2} \bigg( \frac{2\rho^{2}r^{2}}{(r^{2} - \rho^{2})^{4}}  - \frac{4\rho^{2}}{(r^{2} - \rho^{2})^{3}} - \frac{d}{(r^{2}- \rho^{2})^{2}}  \bigg) 
 \\
 &
 \geq - \frac{ 2\sqrt{2}  R }{r^{3} } + \frac{\varsigma^{2}}{2}   R  ^{2} \bigg( - \frac{16}{r^{4}} - \frac{4d}{r^{4}} \bigg) =: c_{2}(R, r)\,.
\end{align*}
For $\rho \in (\frac{r}{\sqrt{2}} , r]$ and $t \in [-R, R]\setminus\{0\}$  we have that 
\begin{align*}
\frac{\varsigma^2}{2}  t^{2} \bigg( \frac{\rho^{2}r^{2}}{(r^{2} - \rho^{2})^{4}}  - \frac{4\rho^{2}}{(r^{2} - \rho^{2})^{3}} - \frac{d}{(r^{2}- \rho^{2})^{2}} \bigg)  \geq c_{3} (R, r)
\end{align*}
for some finite constant~$c_{3} (R, r) \in \R$. For $\rho \in (\frac{r}{\sqrt{2}}, r]$ and $t \in [-R, R]\setminus\{0\}$ we further estimate
\begin{align*}
-\frac{\rho t  }{(r^{2}- \rho^{2})^{2}}  + \frac{\varsigma^{2}}{2} t^{2}  \frac{\rho^{2}r^{2}}{(r^{2} - \rho^{2})^{4}}\geq - \frac{r|t| }{(r^{2}- \rho^{2})^{2}} +  \frac{\varsigma^{2}}{2} t^{2}  \frac{r^{4}}{2(r^{2} - \rho^{2})^{4}} =: j_{r, t} (\rho) \,.
\end{align*}
In particular, we may restrict to $t >0$. By a direct computation, for $t>0$ we find that $j_{r, t} \colon \R \to \R$ has a minimum point in $\overline{\rho}$ satisfying
\begin{displaymath}
(r^{2} - \overline{\rho}^{2})^{2} = \varsigma^{2}\, t\, r^{3} \,.
\end{displaymath}
Substituting in $j_{r, \rho}$ we obtain that
\begin{displaymath}
j_{r, t} (\rho) \geq j_{r, t} ( \overline{\rho} ) = - \frac{1}{\varsigma^{2} r^{2}} + \frac{1}{ 2 \varsigma^{2} r^{2}} =: c_{4} (R, r)\,.
\end{displaymath}
Setting $q(R, r) \coloneqq 2r^{2} \big( c_{1} (R, r) +  \max\{c_{2}(R, r) \ ; \ c_{3} (R, r) + c_{4} (R, r)\} \big) $ we infer~\eqref{e:ineq-phi-r}.
\end{proof}

As a consequence of Lemma~\ref{l:phi-R}, we have the following estimate from below on the measure $\rho^{n}_{t} (B_{r})$ along the solution to system~\eqref{e:system-n}.

\begin{proposition}
\label{e:mass-ball-r}
Assume that $(T1)$--$(T2)$ and $(f1)$--$(f3)$ hold for every $n \in \mathbb{N}$, with a uniform constant~$M_{f_{n}} = M \geq 0$ in $(f2)$. Let $\widehat{X}_0\colon\Omega\to\R^d$ an $\sF_0$-measurable random variable belonging to $L^4(\Omega,\sF,\prob)$, and $\widehat{\Lambda}_0\colon\Omega\to[0,1]$ an $\sF_0$-measurable random variable belonging to $L^1(\Omega,\sF,\prob)$, and let $(X^{n}_{t}, \Lambda^{n}_{t})$ be the solution to system~\eqref{e:system-n} in $[0, +\infty)$ with initial condition $(X^{n}_{0}, \Lambda^{n}_{0} ) = ( \widehat{X}_{0}, \widehat{\Lambda}_{0})$. Then, for every $T>0$ and every $r>0$ there exists a positive constant $\overline{q} = \overline{q}(T, r)>0$ independent of~$n$ such that
\begin{equation}
\label{e:estimate-mass}
\rho^{n}_{t} (B_{r}) \geq \E ( \phi_{r} ( \widehat{X}_{0} ))  e^{-\overline{q} t }  \qquad \text{for $t \in [0, T]$.}
\end{equation}
\end{proposition}

\begin{proof}
Let us fix $T >0$. We apply It\^o equality with test function $\phi_{r}$\,. For every $t \in [0, +\infty)$ we have
\begin{align}
\label{e:estimate-phi-r}
\E ( \phi_{r} ( X^{n}_{t} ) ) = &\ \E ( \phi_{r} ( \widehat{X}_{0} )) + \int_{0}^{t} \E \big( \nabla \phi_{r} ( X^{n}_{s} ) \cdot  v^{n}_{\rho^{n}_{s}} ( X^{n}_{s} , \Lambda^{n}_{s})  \big) \, \de s
\\
&
+ \frac{\varsigma^{2}}{2}  \int_{0}^{t} \E\big( \Delta \phi_{r} ( X^{n}_{s} ) \, \|  v^{n}_{\rho^{n}_{s}} ( X^{n}_{s} , \Lambda^{n}_{s})  \|^{2} \big) \, \de s\,. \nonumber
\end{align}
Let us fix $T>0$. In view of Theorem~\ref{t:well-posedness} and Remark~\ref{r:aremark}, we have that there exists~$C= C(T)>0$ independent of~$n$ such that 
\begin{displaymath}
\E( \|X^{n}_{s} \|)\leq C \qquad \text{and} \qquad \E( \|X^{n}_{s} \|^{2}) \leq C \qquad \text{for $n \in \mathbb{N}$ and $s \in [0, T]$}\,.
\end{displaymath}
In view of $(f2)$, for every $n$, every $s \in [0, T]$, and $\mathbb{P}$-a.e.~$\omega \in \Om$ it holds that 
\begin{align*}
\| \Lambda^{n}_{s} f_{n} ( \rho^{n}_{s} ) + ( 1 - \Lambda^{n}_{s} ) e(\rho^{n}_{s}) \| & \leq \| f_{n} (\rho^{n}_{s}) \| + \| e(\rho^{n}_{s}) \|
 \leq M_{f} \big( 1 + m_{2} (\rho^{n}_{t}) \big)  + C \leq M_{f} ( 1 + C) + C ,
\end{align*}
where $M_{f} \geq 0$ is the constant appearing in~$(f2)$.  Applying Lemma~\ref{l:phi-R} with $R = \sqrt{C_{0} ( 1 + C)} + C$ we may proceed in~\eqref{e:estimate-phi-r} with 
\begin{align}
\label{e:estimate-phi-r-2}
\E ( \phi_{r} ( X^{n}_{t} ) ) \geq  \E ( \phi_{r} ( \widehat{X}_{0} )) - \overline{q} \int_{0}^{t} \phi_{r} (X_{s}^{n}) \, \de s\,.
\end{align}
where $\overline{q} = \overline{q} (R, r)>0$ is the constant in \eqref{e:ineq-phi-r}. Notice that $\overline{q}$ ultimately depends on~$T$ and~$r$, as $C$ only depends on $T$. By Gr\"onwall inequality, we deduce from~\eqref{e:estimate-phi-r-2} that
\begin{displaymath}
\E ( \phi_{r} ( X^{n}_{t} ) ) \geq  \E ( \phi_{r} ( \widehat{X}_{0} )) e^{-\overline{q} t} \qquad \text{for $t \in [0, T]$.}
\end{displaymath}
This yields~\eqref{e:estimate-mass} as $0 \leq \phi_{r} \leq 1$.
\end{proof}

We are now in a position to state and prove the main result of this section.

\begin{theorem}
\label{t:systems are close}
 Let $B \colon \Omega \to C([0,T],\R^d)$ a $d$-dimensional standard Brownian motion defined on a filtered probability space $\bigl(\Omega, \sF, (\sF_t)_{t\in[0,T]},\prob\bigr)$, $\widehat{X}_0\colon\Omega\to\R^d$ an $\sF_0$-measurable random variable belonging to $L^4(\Omega,\sF,\prob)$, and $\widehat{\Lambda}_0\colon\Omega\to [0,1]$ an $\sF_0$-measurable random variable. Assume that $(f1)$--$(f3)$ and $(T1)$--$(T3)$ hold for every $n \in \mathbb{N}$, with constant~$M_{f_{n}} = M \geq 0$ in $(f2)$, that $\varsigma^{2}d <2$, and that 
\begin{align}
\label{e:hp-mass}
\E(\widehat{\Lambda}_{0}) >0\, \qquad\text{and} \qquad  \mu_{0} (B_{r}) >0 \ \text{for every $r>0$}.
\end{align}
Then, for every $\varepsilon>0$ there exist $T_{\varepsilon}>0$ and $n_{\varepsilon}\in \mathbb{N}$ such that
\begin{align}
\label{e:thesis-thm}
\E( \| X^{n}_{T_{\eps}} \|^{2}) \leq \eps \qquad \text{for every $n \geq n_{\eps}$\,.}
\end{align} 
\end{theorem}

\begin{proof}
Let us denote by $(X_{t}, \Lambda_{t})$ the solution to the auxiliary system~\eqref{e:system2} with initial conditions~$(\widehat{X}_{0}, \widehat{\Lambda}_{0})$, and let $\Psi_{t} = \Law (X_{t}, \Lambda_{t})$. Under the assumptions $(T1)$--$(T3)$, we have shown in Theorem \ref{t:concentration2} that $\E( \|X_{t} \|^{2}) \to 0$ as $t \to +\infty$. In particular, for every $\eps > 0$ there exists $T_{\eps}>0$ such that
\begin{displaymath}
\E( \| X_{t} \|^{2}) \leq \frac{\eps}{2} \qquad \text{for every $t \geq {T}_{\eps}$\,.}
\end{displaymath}
Hence, it is enough to show that there exists $n_{\eps} \in \mathbb{N}$ such that
\begin{equation}
\label{e:to-prove}
\E( \| X^{n}_{T_{\eps}} - X_{T_{\eps}} \|^{2}) \leq \frac{\eps}{2}\,.
\end{equation}

Let us estimate~$\E(\| (X^{n}_{t}, \Lambda^{n}_{t})  - ( X_{t}, \Lambda_{t}) \|^{2})$ using It\^o formula. 
For $t \in [0, T_{\eps}]$, we have
\begin{align}
\label{e:conv-ito}
& \E( \| (X^{n}_{t}, \Lambda^{n}_{t})  - ( X_{t}, \Lambda_{t}) \|^{2} ) = 2 \int_{0}^{t} (\Lambda^{n}_{s} - \Lambda_{s} ) ( \cT_{\Sigma^{n}_{s}} (X^{n}_{s} , \Lambda^{n}_{s}) -  \cT_{\Psi_{s}}(X_{s}, \Lambda_{s})) \, \de s  
\\
&
+ 2 \int_{0}^{t} \E \big( ( X^{n}_{s} - X_{s}) \cdot ( X_{s} - X^{n}_{s} + \Lambda^{n}_{s} f_{n} (\rho^{n}_{s}) + ( 1 - \Lambda^{n}_{s}) e(\rho^{n}_{s}) - (1 - \Lambda_{s})  e(\mu_{s}) ) \big) \, \de s \nonumber
\\
&
+ \frac{\varsigma^{2} d }{2}  \int_{0}^{t} \E \Big( \big( \| X^{n}_{s} - \Lambda^{n}_{s} f_{n} (\rho^{n}_{s}) - ( 1 - \Lambda^{n}_{s}) e(\rho^{n}_{s}) \| - \|  X_{s} - (1 - \Lambda_{s})  e(\mu_{s}) ) \| \big)^{2} \Big)  \, \de s \,.\nonumber
\end{align}
We recall that, by Theorem~\ref{t:well-posedness} and by Remark~\ref{r:aremark}, $\E( \| X^{n}_{s}\|^{2})$ and~$\E( \| X_{s}\|^{2})$ are bounded in~$[0, T_{\eps}]$ uniformly with respect to $n$, as we are assuming $M_{f_{n}} = M$ in $(f2)$ for every $n \in \mathbb{N}$. By simple algebraic manipulation, by H\"older and triangle inequalities, and by $(T1)$, we thus deduce that
\begin{align}
\label{e:conv-ito-2}
\E( \|(X^{n}_{t}, &  \Lambda^{n}_{t}) - ( X_{t}, \Lambda_{t}) \|^{2})  \leq   2 L_{\cT} \int_{0}^{t} \E \Big[ |\Lambda^{n}_{s} - \Lambda_{s}| \big(  | \Lambda^{n}_{s} - \Lambda_{s}|  +  \| X^{n}_{s} - X_{s} \| + W_{1} (\Sigma^{n}_{s}, \Psi_{s}) \big) \Big] \, \de s 
\\
&
\quad -2\int_{0}^{t} \E( \|X^{n}_{s} - X_{s} \|^{2}) \, \de s 
 + 2 \int_{0}^{t} \E( \Lambda^{n}_{s} (X^{n}_{s} - X_{s}) \cdot f_{n} (\rho^{n}_{s}) \big) \, \de s  \nonumber
\\
&
\quad + 2 \int_{0}^{t} \big \| \E(  X^{n}_{s} - X_{s} ) \big \|^{2}\, \de s 
 + 2 \int_{0}^{t} \E( ( X^{n}_{s} - X_{s}) \cdot (\Lambda^{n}_{s} - \Lambda_{s}) \E(X^{n}_{s}) ) \, \de s \nonumber
\\
&
\quad + 2 \int_{0}^{t} \E( \Lambda_{s}( X^{n}_{s} - X_{s})) \cdot \E(X^{n}_{s} - X_{s}) \, \de s \nonumber
\\
&
\quad + \frac{\varsigma^{2} d}{2}  \int_{0}^{t} \E\big( \| X^{n}_{s} - X_{s}  - \Lambda^{n}_{s} f_{n} (\rho^{n}_{s}) - ( 1 - \Lambda^{n}_{s}) e(\rho^{n}_{s})  - (1 - \Lambda_{s})  e(\mu_{s}) \|^{2} \big) \, \de s\nonumber
\\
&
\leq  C_{{\eps}} \int_{0}^{t} \E( \| ( X^{n}_{s}, \Lambda^{n}_{s})  - ( X_{s}, \Lambda_{s}) \|^{2}) \, \de s \nonumber
\\
&
\nonumber \quad  + C_{{\eps}} \int_{0}^{t} |\Lambda^{n}_{s} - \Lambda_{s}| \,   \E( \| ( X^{n}_{s}, \Lambda^{n}_{s})  - ( X_{s}, \Lambda_{s}) \| ) \, \de s + C_{\eps} \int_{0}^{t} \| f_{n} (\rho^{n}_{s})\|^{2} \, \di s \,, 
\\
&
\leq \overline{C}_{{\eps}} \bigg(  \int_{0}^{t}  \E( \| ( X^{n}_{s}, \Lambda^{n}_{s})  - ( X_{s}, \Lambda_{s}) \|^{2})   \, \de s  + \int_{0}^{t} \| f_{n} (\rho^{n}_{s})\|^{2} \, \di s \bigg) \,, \nonumber
\end{align}
for some positive constants~$C_{\eps}\,, \overline{C}_{{\eps}}$ only depending on~$T_{\eps}$\,, $d$, $\varsigma$, and the initial condition~$\widehat{X}_{0}$\,. 
By Gr\"onwall inequality,~\eqref{e:conv-ito-2} yields
\begin{align}
\label{e:conv-ito-3}
\E( \|X^{n}_{t} - X_{t} \|^{2}) & \leq \E( \| (X^{n}_{t},  \Lambda^{n}_{t}) - ( X_{t}, \Lambda_{t}) \|^{2}) \leq \overline{C}_{{\eps}} e^{\overline{C}_{{\eps}} T_{\eps}} \int_{0}^{T_{\eps}}  \| f_{n} (\rho^{n}_{s}) \|^{2}  \, \di s  
\end{align}
for every $t \in [0, T_{\eps}]$. In view of Proposition~\ref{e:mass-ball-r} and of assumption~\eqref{e:hp-mass}, we are in a position to use~$(f3)$ with $\ell (r) = \E( \phi_{r} (\widehat{X}_{0}))  e^{-\overline{q} T_{\eps}}$ where $\overline{q} = \overline{q} (T, r) >0$ is defined in~\eqref{e:estimate-mass}. Hence,~$f_{n}(\rho^{n}_{s}) \to 0$ as $n  \to \infty$, uniformly in $s \in [0, T_{\eps}]$. Combining such convergence with~\eqref{e:conv-ito-3}, we infer~\eqref{e:to-prove}.
\end{proof}

As a corollary of Theorem~\ref{t:systems are close}, we get the following control estimate concerning the particle approximation of~\eqref{e:particle-SDE}.

\begin{corollary}
Under the assumptions of Theorem~\ref{t:systems are close}, for every $n, N \in \mathbb{N}$ let $(\widehat{X}^{i}_{0} , \widehat{\Lambda}^{i}_{0}) \in L^{4} (\Om, \sF, \mathbb{P}; \R^{d} \times [0, 1])$ be distributed as $( \widehat{X}_{0} , \widehat{\Lambda}_{0} )$ and let us denote by $(X^{i, n, N} , \Lambda^{i, n, N})$ the solution to
\begin{align}
\label{e:complicated system}
\begin{cases}
\di X^{i, n, N} = v^{n}_{\rho^{n, N}_{t}} (X^{i, n, N}_{t} , \Lambda^{i, n, N}_{t}) \, \di t + \varsigma \| v^{n}_{\rho^{n, N}_{t}} (X^{i, n, N}_{t} , \Lambda^{i, n, N}_{t} ) \| \, \di B_{t}\,,\\[1mm]
\di \Lambda^{i, n, N} = \cT_{\Sigma^{n, N}_{t}} (X^{i, n, N}_{t} , \Lambda^{i, n, N}_{t}) \, \di t\,,\\[1mm]
(X^{i, n, N}_{0}, \Lambda^{i, n, N}_{0}) = (\widehat{X}^{i}_{0} , \widehat{\Lambda}^{i}_{0})\,,\\[1mm]
\Sigma^{n, N} = \frac{1}{N} \sum_{j=1}^{N} \delta_{(X^{i, n, N}, \Lambda^{i, n, N})}\,, \quad \rho^{n, N}= (\pi_{x})_{\#} \Sigma^{n, N}\,.
\end{cases}
\end{align}
Then, for every $\varepsilon>0$ there exist $T_{\varepsilon} >0$ and $n_{\varepsilon} \in \mathbb{N}$ such that for every $n \geq n_{\varepsilon}$ there is $N_{\varepsilon, n} \in \mathbb{N}$ such that
\begin{displaymath}
\E \big( m_{2} ^{2} (\Sigma^{n, N}_{T_{\varepsilon}}) \big) < \varepsilon \qquad \text{for $n \geq n_{\varepsilon}$ and $N \geq N_{\varepsilon, n}$}.
\end{displaymath}
\end{corollary}

\begin{proof}
For $\varepsilon>0$, we fix ${T}_{{\varepsilon}} >0$ and $\bar{n}_{{\varepsilon}} \in \mathbb{N}$ such that
\begin{align}
\label{e:8010}
\E ( \| X^{n}_{\overline{T}_{{\varepsilon}}} \|^{2}) \leq \frac{\varepsilon}{2} \qquad \text{for every $n \geq \bar{n}_{{\varepsilon}}$\,,}
\end{align}
 as done in Theorem~\ref{t:systems are close}. We further notice that the solutions $(X^{i, n, N}, \Lambda^{i, n, N})$ satisfy~\eqref{e:moments-finite} and~\eqref{e:limit-moments} uniformly with respect to $n$, thanks to the assumption $M_{f_{n}} = M$ for every $n \in \mathbb{N}$. In view of Lemma~\ref{l:particle-compactness} and Skorohod representation theorem (cf.~\eqref{e:limit-moments}--\eqref{e:limit-moments-2} and related discussion) we have that $\Sigma^{n, N}_{\overline{T}_{\frac{\varepsilon}{2}}} $ converges to $\Sigma^{n, N}_{\overline{T}_{\frac{\varepsilon}{2}}}$ in $(\cP_{2} (\R^{d}), W_{2})$ as $N \to \infty$ for every $n \in \mathbb{N}$. Furthemore, in view of~\eqref{e:limit-moments} we have that 
 \begin{displaymath}
 \lim_{N \to \infty}\, \E \Big( m_{2}^{2} \big(\Sigma^{n, N}_{\overline{T}_{{\varepsilon}}}  \big) \Big) = \E \Big( m_{2}^{2} \big(\Sigma^{n}_{\overline{T}_{{\varepsilon}}} \big) \Big) \qquad \text{for every $n \in \mathbb{N}$.}
 \end{displaymath}
 This, together with~\eqref{e:8010}, concludes the proof of the corollary.
\end{proof}

\section{Application to Consensus Based Optimization}
\label{s:application}

In this section we show how the applications to CBO~\cite{Carrilloetal, FKR}  fall into our theoretical framework. Let $g\colon \R^{d} \to \R^{d}$ be such that there exists $M_{g}>0$ satisfying 
\begin{equation}
\label{e:linear-g}
     \|g(x) \| \leq M_{g} ( 1 + \|x \|) \qquad \text{for every $x \in \R^{d}$.}
\end{equation}
For $\mu \in \cP_{1} (\R^{d})$ and $n \in \mathbb{N}$ we consider
\begin{align}
\omega_{n} (x) & \coloneqq  e^{-n \cE(x) }  \,, \label{e:def-omega-n}\\
f_{n} (\mu) & \coloneqq \frac{1}{\| \omega_{n}\|_{L^{1}_{\mu}}} \int_{\R^{d}} g(x) e^{-n\cE(x)} \, \de \mu(x)\,. \label{e:def-fn}
\end{align}
Then, the following proposition holds.

\begin{proposition}
\label{p:5.1}
Let $g\colon \R^{d} \to \R^{d}$ be Lipschitz continuous and let $\cE \in C^{1}( \R^{d} ; [0, +\infty))$  be such that 
    \begin{equation}
    \label{e:hp-notrestrictive}
        \cE(0) = \min_{x \in \R^{d}} \cE(x)\,.
    \end{equation}
    and there exist $c_{1} >0$ and $\gamma >0$ such that
    \begin{align}
     |  \nabla \cE(x)| & \leq c_{1} ( 1 + \| x\|^{\gamma}) \qquad \text{for every $x \in \R^{d}$}. \label{e:nabla-E-gamma}
    \end{align}
    Assume moreover that either of the two conditions is fulfilled:
    \begin{itemize}
    \item[$(i)$] There exist $c_{2}\,, c_{3} >0$ and $\nu \geq 0$ such that for every $x \in \R^{d}$
        \begin{align}
    \label{e:growth-E-nu-2}
           c_{2} \| x \|^{\nu} & \leq \cE(x) \leq c_{3} (1 + \| x\|^{\nu})\,,
    \end{align}
      \item[$(ii)$] $g$ is bounded in~$\R^{d}$ and there exist $c_{2}>0$ and $\nu >0$ such that
      \begin{align}
    \label{e:growth-E-nu}
          \cE(x) & \geq c_{2} \| x \|^{\nu}\,,
    \end{align}
    \end{itemize}
    Then, there exist $M\geq 0$ and, for every $n \in \mathbb{N}$ and $R>0$, a constant $L_{n, R}>0$ such that for every $\mu, \mu_{1}, \mu_{2} \in \cP_{1} (\R^{d})$ with $m_{1} (\mu_{i}) \leq R$ for $i = 1, 2$
    \begin{align}
    & \| f_{n} (\mu_{1}) - f_{n} (\mu_{2}) \| \leq L_{n, R} \, W_{1} (\mu_{1}, \mu_{2}) \,, \label{e:f1}\\
   &  \| f_{n} (\mu) \| \leq M ( 1 + m_{1} (\mu))\,. \label{e:f2}
    \end{align}
    In particular, $f_{n}$ satisfies $(f1)$ and $(f2)$ for every $n$, and the constant~$M_{f_{n}}$ in $(f2)$ can be chosen to be independent of~$n \in \mathbb{N}$.
    \end{proposition}

\begin{remark}
    Notice that assumption~\eqref{e:hp-notrestrictive} is not restrictive, as it may be guaranteed up to translation of~$\cE$.
\end{remark}

\begin{proof}[Proof of Proposition~\ref{p:5.1}]
We start working under assumptions $(i)$. Let us fix $n \in \mathbb{N}$ and assume that \eqref{e:nabla-E-gamma} and~\eqref{e:growth-E-nu-2} are satisfied. We show that \eqref{e:f1} holds. Let us fix $R>0$ and let $\mu_{1}, \mu_{2} \in \cP_{1} (\R^{d})$ be such that $m_{1} (\mu_{i}) \leq R$ for $i = 1, 2$. We write
\begin{align}
\label{e:ff1}
 \| f_{n} (\mu_{1}) - f_{n} (\mu_{2}) \| & \leq \bigg\| \frac{1}{\| \omega_{n}\|_{L^{1}_{\mu_{1}}}} \int_{\R^{d}} g(x) e^{-n\cE(x)} \, \de (\mu_{1} - \mu_{2}) \bigg\| 
 \\
 &
 \qquad + \frac{ \big| \| \omega_{n}\|_{L^{1}_{\mu_{1}}} - \| \omega_{n}\|_{L^{1}_{\mu_{2}}} \big|}{\| \omega_{n}\|_{L^{1}_{\mu_{1}}} \, \| \omega_{n}\|_{L^{1}_{\mu_{2}}}} \int_{\R^{d}} g(x) e^{-n\cE(x)} \, \de \mu_{2} (x) 
 =: \mathrm{I} + \mathrm{II} \,. \nonumber
 \end{align}
 We notice that the map $e^{-n \cE(x)}$ is Lipschitz continuous. Indeed, by~\eqref{e:growth-E-nu}--\eqref{e:nabla-E-gamma} we notice that
 \begin{displaymath}
 \| \nabla (e^{-n \cE(x)}) \| =  n e^{-n \cE(x)} \|\nabla \cE(x) \| \leq c_{2} n e^{-n c_{1} \| x\|^{\nu}} \big( 1 + \| x\|^{\gamma})\,.
 \end{displaymath} 
 Thus, $ \|\nabla (e^{-n \cE(x)}) \|$ is bounded in~$\R^{d}$ for every $n$. As a consequence, the map $g(x) \, e^{-n \cE(x)}$ is Lipschitz continuous in~$\R^{d}$. Furthermore, we have that
 \begin{displaymath}
 \| \omega_{n} \|_{L^{1}_{\mu_{1}}} \geq \int_{B_{2R}} e^{-n \cE(x)} \, \de \mu_{1} \geq \frac{1}{2} e^{-n \cE_{2R}}\,,
 \end{displaymath}
 where we have set $\cE_{2R} \coloneqq \max_{\overline{B}_{2R}} \cE$ and we have used that $\mu_{1} (B_{2R}) \geq \frac{1}{2}$, as a consequence of Chebyschev inequality. Notice that the same inequality holds for~$\mu_{2}$. Thus, we infer that
 \begin{equation}
 \label{e:ff2}
 \mathrm{I} \leq 2 e^{n \cE_{2R}} {\rm Lip} (g e^{-n \cE}) W_{1} (\mu_{1}, \mu_{2})\,.
 \end{equation}
Since $g$ is Lipschitz continuous, there exists $M_{g} \geq 0$ such that $\|g(x)\| \leq M_{g} ( 1 + \| x\|)$. Hence, we estimate
\begin{equation}
\label{e:ff3}
   \mathrm{II} \leq 4 e^{2n \cE_{2R}}  M_{g} \Big( 1 + \sup_{x \in \R^{d}} (  \|x \| e^{-n \cE(x)} ) \Big)  {\rm Lip} ( e^{-n \cE}) W_{1} (\mu_{1}, \mu_{2})\,.
\end{equation}
Combining~\eqref{e:ff1}--\eqref{e:ff3} we infer~\eqref{e:f1}.

The proof of~\eqref{e:f2} is an adaptation of the arguments of~\cite[Lemma~3.3]{Carrilloetal}. Let us fix~$\mu \in \cP_{1} (\R^{d})$. To shorten the notation, let us set $\eta_{n}\coloneqq \frac{ \omega_{n} \mu }{ \| \omega_{n} \|_{L^{1}_{\mu}}}$. We notice that it holds
\begin{align}
    \label{e:etan1}
    \| f_{n} (\mu) \| \leq \int_{\R^{d}} \| g(x) \|  \, \di \eta_{n} (x) \leq M_{g} \int_{\R^{d}} ( 1 + \| x \|) \, \di \eta_{n} (x) \leq M_{g} \big( 1 + m_{1} (\eta_{n})\big)   \,. 
\end{align}
Hence, we have to estimate $m_{1} (\eta_{n})$ in terms of $m_{1} (\mu)$. For every $k, \ell \in \R$ we have that
\begin{align}
\label{e:etan2}
   \eta_{n} ( \{ \cE \geq k \} )  & = \frac{1}{ \| \omega_{n} \|_{ L^{1}_{\mu}}} \int_{ \{ \cE \geq k \} } \omega_{n} (x) \, \di \mu(x) \leq e^{-nk} \, \frac{ \mu ( \{\cE \geq k\} ) }{ \| \omega_{n} \|_{ L^{1}_{\mu} } } 
   \leq e^{-n(k - \ell) } \frac{ \mu ( \{ \cE \geq k \})}{ \mu (\{ \cE < \ell \} )}\,. 
\end{align}
For every $S>0$ we write
\begin{align}
    \label{e:etan3}
    m_{1} (\eta_{n}) \leq S \, \eta_{n} (B_{S}) + \int_{\R^{d} \setminus B_{S}} \| x \| \, \di \eta_{n} = S + \int_{\R^{d}\setminus B_{S}}  ( \| x \| - S) \, \di \eta_{n}\,.
\end{align}
We introduce $A_{i}\coloneqq \{ x \in \R^{d}: \, 2^{i-1} S \leq \|x \| < 2^{i} S\}$. By the growth conditions~\eqref{e:growth-E-nu} we have that $A_{i} \subseteq \{ \cE \geq c_{2} 2^{\nu(i-1)} S^{\nu}\}$. Hence, we continue in~\eqref{e:etan3} with
\begin{align}
    \label{e:etan4}
     m_{1} (\eta_{n})  & \leq S + \sum_{i=1}^{\infty} \int_{A_{i}} ( \| x \| - S) \, \di \eta_{n} \leq S + S \sum_{i=1}^{\infty} ( 2^{i} - 1) \eta_{n} (A_{i}) 
     \\
     &
     \nonumber \leq S + S \sum_{i=1}^{\infty} ( 2^{i} - 1) \, \eta_{n} ( \{ \cE \geq c_{2} 2^{\nu(i-1)} S^{\nu} \})
     \\
     &
     \nonumber\leq S + S \sum_{i=1}^{\infty} ( 2^{i} - 1) e^{-n c_{2} 2^{\nu (i-2)} ( 2^{\nu} - 1) S^{\nu}} \frac{ \mu (\{ \cE \geq c_{2} 2^{\nu(i-1)} S^{\nu} \} )}{\mu (\{ \cE < c_{2} 2^{\nu(i-2)} S^{\nu}\} ) }\,,
\end{align}
    where in the last inequality we have used~\eqref{e:etan2} with $k = c_{2} 2^{\nu(i-1)} S^{\nu}$ and $\ell = c_{2} 2^{\nu(i-2)} S^{\nu}$.

By Chebyschev inequality and by~\eqref{e:growth-E-nu} we further estimate for $k >0$
\begin{align}
    \label{e:estimate-mu-example}
    \mu ( \{ \cE \geq k \}) & = \mu ( \{ \cE^{\frac{1}{\nu}} \geq k^{\frac{1}{\nu}} \}) \leq \frac{1}{k^\frac{1}{\nu}} \int_{\R^{d}} \cE^{\frac{1}{\nu}} (x) \, \di \mu \leq \frac{c_{3}^{\frac{1}{\nu}}}{{k^\frac{1}{\nu}}} \int_{\R^{d}} ( 1 + \| x \|^{\nu}) ^{\frac{1}{\nu}} \, \di \mu
    \\
    &
    \leq \frac{\overline{c}}{k^\frac{1}{\nu}} \int_{\R^{d}} (1 + \| x \|) \, \di \mu = \frac{\overline{c}}{k^{\frac{1}{\nu}}} ( 1 + m_{1} (\mu))\,, \nonumber
\end{align}
for a positive constant~$\overline{c}>0$ only depending on~$c_{3}$ and on $\nu$. Inequality~\eqref{e:estimate-mu-example} implies that
\begin{align}
    \label{e:estimate-mu-example2}
    \mu \bigg( \bigg\{ \cE < \frac{k}{2^{\nu}} \bigg\}\bigg) & = 1 - \mu \bigg( \bigg\{ \cE \geq \frac{k}{2^{\nu}} \bigg\}\bigg) \geq 1 - \frac{2\overline{c}}{{k^\frac{1}{\nu}}} ( 1 + m_{1} (\mu))  = \frac{k^\frac{1}{\nu} - 2\overline{c}( 1 + m_{1} (\mu)) }{k^\frac{1}{\nu}}\,.
\end{align}
Combining~\eqref{e:estimate-mu-example}--\eqref{e:estimate-mu-example2} with $k = c_{2} 2^{\nu(i-1)} S^{\nu}$ and inserting in~\eqref{e:etan4}, we obtain
\begin{align}
    \label{e:etan5}
    m_{1} (\eta_{n}) \leq S + S\sum_{i=1}^{\infty} ( 2^{i} - 1) e^{-n c_{2} 2^{\nu (i-2)} ( 2 ^{\nu} - 1) S^{\nu}} \, \frac{\overline{c} ( 1 + m_{1} (\mu))}{c_{2}^{\frac{1}{\nu}} 2^{(i-1)} S  - 2\overline{c}( 1 + m_{1} (\mu)) }\,.
\end{align}
We fix $S = 1 + 2c_{2}^{-\frac{1}{\nu}} \overline{c} ( 1 + m_{1} (\mu))$ and continue in~\eqref{e:etan5} with
\begin{align}
    \label{e:etan6}
     m_{1} (\eta_{n}) &  \leq S + S\sum_{i=1}^{\infty} ( 2^{i} - 1) e^{-n c_{2} 2^{\nu (i-2)} (2^{\nu} - 1) S^{\nu}} \frac{\overline{c} ( 1 + m_{1} (\mu))}{c_{2}^{\frac{1}{\nu}} 2^{(i-1)}   + (2^{i} - 2)\overline{c}( 1 + m_{1} (\mu)) }
     \\
     &
     \leq S + S\sum_{i=1}^{\infty} ( 2^{i} - 1) e^{-n c_{2} 2^{\nu (i-2)} (2^{\nu}-1) S^{\nu}} \frac{\overline{c} ( 1 + m_{1} (\mu))}{c_{2}^{\frac{1}{\nu}} 2^{(i-1)} }\nonumber 
     \\
     &
     = S + \frac{2 \overline{c} ( 1 + m_{1} (\mu)) S}{c_{2}^{\frac{1}{\nu}} } \sum_{i=1}^{\infty} ( 1 - 2^{-i} ) e^{-n c_{2} 2^{\nu (i-2)} (2^{\nu} - 1)  S^{\nu}}  \nonumber
     \\
     &
     \leq  S + \frac{2 \overline{c} ( 1 + m_{1} (\mu)) S}{c_{2}^{\frac{1}{\nu}} } \int_{2^{-\nu}}^{+\infty} e^{-n c_{2} (2^{\nu} - 1)  S^{\nu} \sigma } \, \di \sigma
     = S + \frac{2 \overline{c} ( 1 + m_{1} (\mu)) S}{c_{2}^{\frac{1}{\nu}} } \, \frac{e^{-n c_{2} ( 1 - 2^{-\nu}) S^{\nu}}}{n c_{2} (2^{\nu} - 1)  S^{\nu}} \nonumber
     \\
     &
     = S + \frac{2 \overline{c} ( 1 + m_{1} (\mu)) S^{1-\nu}}{n (2^{\nu} - 1) c_{2}^{1 + \frac{1}{\nu}} } \, e^{-n c_{2} ( 1 - 2^{-\nu})  S^{\nu}} \nonumber\,.
\end{align}
We notice that the map $t \mapsto t^{1-\nu} e^{-n c_{2}^{\nu} ( 1 - 2^{-\nu})  t^{\nu}}$ is bounded in $[1, +\infty)$, uniformly with respect to $n$. Since $S \geq 1$, we infer from~\eqref{e:etan6} that there exists $\widetilde{c}>0$ independent of~$n$ such that
\begin{displaymath}
    m_{1} (\eta_{n}) \leq \widetilde{c} \, (1 + m_{1}(\mu))\,.
\end{displaymath}
Together with~\eqref{e:etan1}, this concludes the proof of~\eqref{e:f2}.

Let us now assume that $(ii)$ holds. Then,~\eqref{e:f2} follows immediately from the boundedness of~$g$ with the choice $M  = \| g\|_{\infty}$. The proof of~\eqref{e:f1} works as in the case $(i)$. By the boundedness of $g$ we have that $ \| g(x) e^{-n \cE(x)} \| \leq \| g\|_{\infty}$ for every $n$. We can simplify the estimate of $\mathrm{II}$ in~\eqref{e:ff3} with
 \begin{equation}
 \label{e:II-example}
 \mathrm{II} \leq 4 e^{2n \cE_{2R}} \| g\|_{\infty} {\rm Lip} ( e^{-n \cE}) W_{1} (\mu_{1}, \mu_{2})\,;
 \end{equation} 
 the proof is concluded.
 \end{proof}

The next proposition deals instead with property $(f3)$ introduced in Section~\ref{s:convergence}.

\begin{proposition}
\label{p:f3-example}
Let $g\colon \R^{d} \to \R^{d}$ satisfy
\begin{displaymath}
    \| g(x) \| \leq M_{g}\,  \| x \| \qquad \text{for every $x \in \R^{d}$}
\end{displaymath}
for some $M_{g}>0$ and let $\cE\colon \R^{d}\to [0, +\infty)$ be continuous, fulfill~\eqref{e:hp-notrestrictive}, and such that there exist~$c, \nu >0$ with
\begin{equation}
\label{e:lower-bound}
   c \, \| x \|^{\nu} \leq \cE(x) \qquad \text{for every $x \in \R^{d}$}.
\end{equation}
Then, the sequence~$f_{n}$ defined in~\eqref{e:def-omega-n}--\eqref{e:def-fn} complies with~$(f3)$.
\end{proposition}

\begin{proof}
Let us fix $\ell \colon (0, +\infty) \to (0, 1)$ and $\eps>0$. Since $\cE$ is continuous and $\cE(0) = 0$, there exists $r_{\eps} >0$ such that 
\begin{displaymath}
\cE_{r_{\eps}} \coloneqq \max_{\overline{B}_{r_{\eps}}} \, \cE < \frac{\eps}{2}\,.
\end{displaymath}
Let us further denote by $\theta_{\eps}\coloneqq \big( \frac{ \eps }{c}\big)^{\frac{1}{\nu}}$, where $c$ and~$\nu$ are the parameters appearing in~\eqref{e:lower-bound}. In particular, notice that for every $x \in \R^{d} \setminus \overline{B}_{\theta_{\eps}}$ it holds that
\begin{equation}
\label{e:useful}
\cE(x) \geq c \, \| x \|^{\nu} \geq c \, \theta_{\eps}^{\nu} \geq \eps > \frac{\eps}{2} + \cE_{r_{\eps}}\,.
\end{equation}
Then, for every $\mu \in E_{\ell, R}$ we estimate
\begin{align}
\label{e:fn-to-0}
\| f_{n}(\mu) \| & \leq \frac{1}{\| \omega_{n}\|_{L^{1}_{\mu}} } \int_{\R^{d}} \| g(x) \| e^{-n\cE(x) } \, \di \mu(x) 
\\
&
= \frac{1}{\| \omega_{n}\|_{L^{1}_{\mu}}} \int_{B_{\varrho_{\eps}}} \| g(x) \| e^{-n\cE(x) } \, \di \mu(x) + \frac{1}{\| \omega_{n}\|_{L^{1}_{\mu}}} \int_{\R^{d} \setminus B_{\varrho_{\eps}}} \| g(x) \| e^{-n\cE(x) } \, \di \mu(x)  \nonumber
\\
&
\leq  M_{g} \varrho_{\eps} + \frac{ M_{g} \,  e^{n\cE_{r_{\eps}}}}{\ell(r_{\eps})}  \int_{\R^{d} \setminus B_{\varrho_{\eps}}} \| x \| e^{-n c \| x \|^{\nu}} \, \de \mu(x)\,. \nonumber
\end{align}
We notice that the map $t \mapsto t e^{-n c \| x \|^{\nu}}$ reaches its maximum in $t_{n} \in (0, +\infty)$ and $t_{n} \to 0$ as $n \to \infty$. Hence, for $n$ sufficiently large (depending on~$\varepsilon$) we have that $t_{n}<\theta_{\eps}$ and
\begin{displaymath}
    \max_{\R^{d} \setminus B_{\theta_{\eps}}} \, \| x \| e^{-n c \| x \|^{\nu}} = \varrho_{\eps} e^{-n c \theta_{\eps}^{\nu}}\,.
\end{displaymath}
Hence, we continue in~\eqref{e:fn-to-0} with
\begin{align*}
\| f_{n}(\mu) \| & \leq M_{g} \theta_{\eps} +  \frac{ M_{g} \, \theta_{\eps} e^{n\cE_{r_{\eps}}}}{\ell(r_{\eps})} e^{-n c \theta_{\eps}^{\nu}} \leq \bigg( M_{g}  + \frac{M_{g} \,e^{-n \frac{\eps}{2}}}{\ell(r_{\eps})} \bigg) \bigg(  \frac{ \eps }{c} \bigg) ^{\frac{1}{\nu}}\,,
\end{align*}
where, in the last inequality, we have used~\eqref{e:useful}. This yields $(f3)$ for the sequence~$f_{n}$.
\end{proof}

\bigskip
\noindent\textbf{Acknowledgments}
All the authors are members of the GNAMPA group of INdAM.
The work of S.A. and F.S.~is part of the MUR - PRIN 2022, project \emph{Variational Analysis of Complex Systems in Materials Science, Physics and
Biology}, No. 2022HKBF5C, funded by European Union Next Generation EU, and further supported by Gruppo Nazionale per l'Analisi Matematica, la Probabilit\`a e le loro Applicazioni (GNAMPA-INdAM, Project 2025: DISCOVERIES - \emph{Difetti e Interfacce in Sistemi Continui: un'Ottica Variazionale in Elasticit\`{a} con Risultati Innovativi ed Efficaci Sviluppi}). S.A.~also aknowledges the support of the the FWF Austrian Science Fund through the Project 10.55776/P35359 and of the University of Naples Federico II through the FRA Project "ReSinApas".
The work of M.M.~was in part carried out within the MUR - PRIN 2022 project \emph{Geometric-Analytic Methods for PDEs and Applications (GAMPA)}, ref.~2022SLTHCE – cup E53D2300588 0006 - funded by European Union - Next Generation EU. 
This manuscript reflects only the authors’ views and opinions and the Ministry cannot be considered responsible for them”. This research fits within the scopes of the GNAMPA 2022 Project \emph{Approccio multiscala all’analisi di modelli di interazione}.

\bibliographystyle{siam}

\end{document}